\documentclass[reqno]{amsart}
\usepackage{amsfonts,amsthm,amsmath,amssymb}
\usepackage{mathtools}
\mathtoolsset{showonlyrefs} 

\usepackage{aliascnt} 

\usepackage[dvipsnames]{xcolor}

\usepackage{tabularx} 
\usepackage{empheq}

\theoremstyle{plain}

\newtheorem{Th}{Theorem}[section]

\newtheorem*{Th*}{Theorem}

\newaliascnt{Lemma}{Th}
\newtheorem{Lemma}[Lemma]{Lemma}
\aliascntresetthe{Lemma}

\newtheorem*{Lemma*}{Lemma}

\newaliascnt{Prop}{Th}
\newtheorem{Prop}[Prop]{Proposition}
\aliascntresetthe{Prop}

\newtheorem*{Prop*}{Proposition}

\newaliascnt{Cor}{Th}

\aliascntresetthe{Cor}

\newtheorem*{Cor*}{Corollary}

\newaliascnt{conj}{Th}

\aliascntresetthe{conj}

\newtheorem*{conj*}{Conjecture}

\theoremstyle{definition}

\newaliascnt{Def}{Th}
\newtheorem{Def}[Def]{Definition}
\aliascntresetthe{Def}

\newtheorem*{Def*}{Definition}

\newaliascnt{Ex}{Th}

\aliascntresetthe{Ex}

\newtheorem*{Ex*}{Example}

\newaliascnt{Conj}{Th}

\aliascntresetthe{Conj}

\newtheorem*{Conj*}{Conjecture}

\theoremstyle{remark}

\newaliascnt{Rem}{Th}
\newtheorem{Rem}[Rem]{Remark}
\aliascntresetthe{Rem}

\newtheorem*{Rem*}{Remark}
\let\temp\phi
\let\phi\varphi
\let\varphi\temp

\let\temp\epsilon
\let\epsilon\varepsilon
\let\varepsilon\temp

\let\subset\subseteq

\DeclarePairedDelimiter{\abs}{\lvert}{\rvert}
\DeclarePairedDelimiter{\ang}{\langle}{\rangle}
\DeclarePairedDelimiterX{\ip}[2]{\langle}{\rangle}{#1,#2}
\DeclarePairedDelimiter{\bra}{[}{]}
\DeclarePairedDelimiter{\norm}{\lVert}{\rVert}
\DeclarePairedDelimiter{\p}{(}{)}
\DeclarePairedDelimiter{\set}{\{}{\}}

\newcommand{\R}{\mathbb{R}}

\newcommand{\Z}{\mathbb{Z}}

\newcommand{\ocal}{\mathcal{O}}

\newcommand{\gtc}[1]{G^{#1}_\tau}

\newcommand{\pmt}{\partial M_{\tau}}

\newcommand{\pdp}{\Pi_\tau D_{\sqrt{\rho}} \Pi_\tau}
\newcommand{\pcl}{\Pi_{\chi, \lambda}}
\newcommand{\Pcl}{P_{\chi, \lambda}}

\newcommand{\ptau}{\phi_\tau}

\newcommand{\tl}{\frac{\theta}{\lambda}}
\newcommand{\phil}{\frac{\varphi}{\lambda}}
\newcommand{\vl}{\frac{v}{\sqrt{\lambda}}}
\newcommand{\ul}{\frac{u}{\sqrt{\lambda}}}

\newcommand{\quaddifferential}{ d\sigma_1 d\sigma_2 d\mu_{\tau}(w) dt}

\DeclareMathOperator{\Imp}{Im}
\DeclareMathOperator{\Rep}{Re}
\DeclareMathOperator{\supp}{supp}

\newcommand{\dtilde}[1]{\tilde{\raisebox{0pt}[0.9\height]{$\tilde{#1}$}}}
\usepackage[unicode=true]{hyperref}
\hypersetup{
    colorlinks,
    linkcolor={red!65!black},
    citecolor={purple!40!black},
    urlcolor={blue!80!black}
}

\let\C\relax
\newcommand{\C}{\mathbb{C}} 

\begin{document}

\title{Scaling asymptotics for Szeg\H{o} kernels on Grauert tubes}

\author{Robert Chang}
\address[Robert Chang]{Northeastern University, Boston, MA 02115}
\email{hs.chang@northeastern.edu}

\author{Abraham Rabinowitz}\footnote{A.R. is partially supported by NSF RTG grant DMS-1502632}
\address[Abraham Rabinowitz]{Northwestern University, Evanston, IL 60208, USA}
\email{arabin@math.northwestern.edu}

\begin{abstract}
   Let $M_\tau$ be the Grauert tube of radius $\tau$ of a closed, real analytic Riemannian manifold $M$. Associated to the Grauert tube boundary is the orthogonal projection $\Pi_\tau \colon L^2(\partial M_\tau) \to H^2(\partial M_\tau)$, called the Szeg\H{o} projector. Let $D_{\sqrt{\rho}}$ denote the Hamilton vector field of the Grauert tube function $\sqrt{\rho}$ acting as a differential operator. We prove scaling asymptotics for the spectral localization kernel of the Toeplitz operator $\Pi_\tau D_{\sqrt{\rho}} \Pi_\tau$. We also prove scaling asymptotics for the smoothed spectral projection kernel $\Pcl(z,w) = \sum_{\lambda_j \le \lambda}\chi(\lambda - \lambda_j) e^{-2\tau\lambda_j} \phi_{\lambda_j}^\C(z) \overline{\phi_{\lambda_j}^\C(w)}$, where $\phi_{\lambda_j}^\C$ are CR holomorphic functions on the Grauert tube boundary $\partial M_\tau$, which are obtained by analytically continuing Laplace eigenfunctions on $M$.
\end{abstract}

\maketitle

\tableofcontents

\section{Statement of main results}

Let \((M,g)\) be a closed, real analytic Riemannian manifold. Its Grauert tube \(M_\tau\) is a K\"{a}hler manifold with boundary that is diffeomorphic to the co-ball bundle consisting of co-vectors of length at most \(\tau\). The \emph{Szeg\H{o} projector} \(\Pi_\tau\) associated to the boundary of a Grauert tube, that is, the orthogonal projection
\begin{equation}
\Pi_\tau \colon L^2(\partial M_\tau) \to H^2(\partial M_\tau)
\end{equation}
onto the space of CR holomorphic functions on the Grauert tube boundary that are square integrable.

Fix a positive, even Schwartz function $\chi$ whose Fourier transform is compactly supported with \( \hat{\chi}(0) = 1\). We study the smoothed out spectral localizations
\begin{equation}\label{eqn:PCL}
\Pi_{\chi,\lambda} := \Pi_\tau \chi(\Pi_\tau D_{\sqrt{\rho}} \Pi_\tau - \lambda) =  \int_\R \hat{\chi}(t) e^{-it\lambda} \Pi_\tau e^{it \Pi_\tau D_{\sqrt{\rho}}\Pi_\tau}\,dt
\end{equation}
of the Toeplitz operator
\begin{equation}\label{eqn:TOEPLITZ}
\Pi_\tau D_{\sqrt{\rho}} \Pi_\tau \colon H^2(\partial M_\tau) \to H^2(\partial M_\tau),
\end{equation}
where $D_{\sqrt{\rho}} = \frac{1}{i} \Xi_{\sqrt{\rho}}$ is a constant multiple of the Hamilton vector field of the Grauert tube function $\sqrt{\rho}$ acting as a differential operator; see \autoref{sec:GTUBE} for background on Grauert tubes. As explained in \autoref{sec:COMPARE}, the operator \eqref{eqn:PCL} is an analogue of the Fourier components \eqref{eqn:SPECTRALDECOMP} of the Szeg\H{o} kernel in the line bundle setting.

The scaling asymptotics is expressed in \emph{Heisenberg coordinates} on $\partial M_\tau$ in the sense of \cite[Section~14, 18]{FollandStein}. We give the precise definition of the coordinates here; additional geometric background is recalled in \autoref{sec:HEISENBERG}.

\begin{Def}\label{defn:COORDINATES}
By \emph{Heisenberg coordinates} (relative to the orthonormal frame $Z_1, \dotsc, Z_{m-1}$) at $T^{1,0}_p\partial M_\tau$, we mean holomorphic coordinates $z_0, \dotsc, z_{m-1}$ in a neighborhood of $p \in M_\tau$ with the following properties.
\begin{itemize}
\item[(i)] If we set $\theta = \Rep z_0$ and $z_j = x_j + i y_j$, then $x_1, \dotsc x_{m-1}, y_1, \dotsc, y_{m-1}, \theta$ form a coordinate system in a neighborhood of $p$ in $\partial M_\tau$.

\item[(ii)] Moreover, we have
\begin{align}
Z_j &= \frac{\partial}{\partial z_j} + i \overline{z}_j \frac{\partial}{\partial  t} + \sum_{k = 1}^{m-1} O^1 \frac{\partial}{\partial z_k} + O^2 \frac{\partial}{\partial \theta},\\
T &= \frac{\partial}{\partial \theta} + \sum_{k=1}^{m-1} O^1  \frac{\partial}{\partial z_k} + \sum_{k=1}^{m-1} O^1\frac{\partial}{\partial \overline{z}_k} + O^1 \frac{\partial}{\partial \theta}.
\end{align}
Here, the Heisenberg-type order $O^k$ is defined inductively by
\begin{align}\label{eqn:HEISENBERGORDER}
f &= O^1 \text{ if $f(\eta) = O\p*{\sum_{k=1}^{m-1} (\abs{x_k(\eta)} + \abs{y_k(\eta)}) + \abs{\theta(\eta)}^\frac{1}{2}}$ as $\eta \to p$},\\
f &= O^k \text{ if $f = O(O^1 \cdot O^{k-1})$}.
\end{align}
\end{itemize}
\end{Def}

\begin{Th}[Scaling asymptotics for $\Pi_{\chi,\lambda}$]\label{Scaling Theorem}
		Let $\pcl$ be the spectral projection \eqref{eqn:PCL} associated to the Grauert tube boundary $\partial M_\tau$ of a closed, real analytic Riemannian manifold $M$ of dimension $m$. Fix \(p \in \pmt\). Then, in Heisenberg coordinates centered at $p = 0$,  the distribution kernel has the following scaling asymptotics:
		\begin{align}\label{eqn:SCALINGTHEOREM1}
		\Pi_{\chi, \lambda}&\left( \frac{\theta}{\lambda}, \frac{u}{\sqrt{\lambda}}; \frac{\varphi}{\lambda}, \frac{v}{\sqrt{\lambda} } \right) \\
		&= \frac{C_{m,M}}{\tau}\left( \frac{\lambda }{\tau} \right)^{m - 1}e^{\frac{1}{\tau} \left(\frac{i}{2}\left( \theta - \varphi \right) - \frac{|u|^{2}}{2} - \frac{|v|^{2}}{2} +  v \cdot \overline{u} \right)}\\
         & \quad \times\left(1 + \sum_{j = 1}^{N}\lambda^{ - \frac{j}{2}}P_j(p,u,v,\theta,\varphi) +  \lambda^{^{ - \frac{N + 1}{2}}} R_N\left(p, \theta, u , \varphi, v ,\lambda \right) \right),
        \end{align}
        where \(\norm{R_{N}\left( p, \theta, u, \varphi, v, \lambda \right)}_{C^{j}\left( \{\abs{\theta} + \abs{\varphi} + |u| + |v| \leq \rho\}  \right)} \leq C_{K,j, \rho}\) for \(\rho > 0, \ j = 1,2,3, \dotsc \) and \(P_j\) is a polynomial in \(u,v, \theta, \varphi\).
\end{Th}

To give some context to our scaling asymptotics, recall the reduced Heisenberg group \(\mathbf{H}^{m-1}_{\mathrm{red}} = \mathbf{H}^{m-1} / \{( 0, 2 \pi k ) : k \in \Z\}  = S^1 \times \C^{m-1} \) of degree $m-1$, which is obtained as a discrete quotient of the Heisenberg group $\mathbf{H}^{m-1} = \R \times \C^{m-1}$. The group law on $\mathbf{H}^{m-1}_{\mathrm{red}}$ is given by
\begin{align}
    ( e^{it},\zeta ) \cdot (  e^{is},\eta ) = \p[\Big]{ e^{i ( t + s + \Imp ( \zeta \cdot \overline{\eta}  ) )},\zeta + \eta }.
\end{align}
The level one Szeg\H{o} projector \(\Pi^{\mathbf{H}}_{1} \colon L^{2}( \mathbf{H}^{m-1}_{\mathrm{red}}) \to H^{2}_{1}\) is the orthogonal projection onto the space $H^2_1$ of CR holomorphic functions square integrable with respect to \(e^{ - |z|^{2}}\). Its Schwartz kernel has an exact formula:
\begin{align}\label{eqn:REDUCEDHEISENBERG}
    \Pi_{1}^{\mathbf{H}}( \theta, z, \phi, w ) & = \frac{1}{\pi^{m-1} }e^{i \left( \theta - \phi \right)}e^{ z \cdot \overline{w} - \frac{1}{2}|z|^{2} - \frac{1}{2}|w|^{2}}.
\end{align}
We refer to \cite{BleherShiffmanZelditch00} for more details. It is known  (cf \eqref{eqn:LINEBUNDLE}) that, in correctly chosen coordinates, the Szeg\H{o} kernel in the line bundle case is to leading order some multiple of \eqref{eqn:REDUCEDHEISENBERG}; see for example \cite[Theorem~3.1]{BleherShiffmanZelditch00} or \cite[Theorem~2.1]{ShiffmanZelditch02}. \autoref{Scaling Theorem} shows that the same phenomenon holds in the Grauert tube setting:
\begin{align}
						\Pi_{\chi, \lambda}&\left( \frac{\theta}{\lambda}, \frac{u}{\sqrt{\lambda}}; \frac{\varphi}{\lambda}, \frac{v}{\sqrt{\lambda} } \right)=\frac{C_{m,M}}{\tau}\left( \frac{\lambda }{\tau} \right)^{m - 1} \Pi^{\mathbf{H}}_1\left( \frac{\theta}{2 \tau}, \frac{u}{\sqrt{\tau} }, \frac{\varphi}{2\tau}, \frac{v}{\sqrt{\tau} }  \right)\\
						&\quad \times \left(1 + \sum_{j = 1}^{N}\lambda^{ - \frac{j}{2}}P_j(p,u,v,\theta,\varphi) +  \lambda^{^{ - \frac{N + 1}{2}}} R_N\left(p, \theta, u , \varphi, v ,\lambda \right) \right) 
\end{align}

Another natural operator to study is the spectral projection kernel defined using analytic continuation of Laplace eigenfunctions. Let
\begin{equation}
    \Delta = -\frac{1}{\sqrt{\lvert g \rvert}} \frac{\partial}{\partial x_j}\bigg(\sqrt{\lvert g \rvert} g^{ij} \frac{\partial}{\partial x_k}\bigg)
\end{equation}
be the (positive) metric Laplacian. It is well-known that \(\Delta\) has a discrete spectrum:
\begin{align}
    0 = \lambda_0^{2} < \lambda_1^{2} \leq \dots \lambda_j^{2} \leq \dots  \to \infty
\end{align}
with associated $L^2$-normalized eigenfunctions \begin{align}\label{eqn:EIGENEQUATION}
    \Delta \phi_{\lambda_j} = \lambda_j^{2}\phi_{\lambda_j}, \quad  \|\phi_{\lambda_j} \| _{L^{2}(M)} = 1.
\end{align}
Our next result concerns the complex analogue of the partial spectral projection kernel
\begin{align}\label{eqn:REALSPECTRALPROJ}
    E_{\lambda}(x,y) & = \sum_{\lambda_j \leq \lambda} \phi_{\lambda_j}(x) \overline{\phi_{\lambda_j}(y)}.
\end{align}

The Poisson wave operator is used to analytically continue eigenfunctions. One starts with the half-wave operator $U(t) = e^{ i t \sqrt{\Delta}} \colon L^2(M) \to L^2(M)$ whose Schwartz kernel can be expressed as a sum of eigenfunctions as well as an oscillatory integral (for instance the Lax--H\"{o}rmander parametrix):
\begin{equation}\label{eqn:WAVEOPERATOR}
U(t,x,y) = \sum_{j = 0}^\infty e^{it \lambda_j}  \phi_{\lambda_j}(x) \overline{\phi_{\lambda_j}(y)} = \int_{T^*_yM} A(t,x,y,\xi)e^{it \abs{\xi}_y} e^{i \ang{\xi,\exp_y^{-1}(x)}}\,d\xi.
\end{equation}
Analytically continuing the oscillatory integral representation in the time variable $t \mapsto t + i\tau \in \C$ and the spatial variable $x \to z \in \partial M_\tau$ yields the Poisson wave kernel (see \autoref{theo:POISSON}):
\begin{equation}\label{eqn:POISSONWAVEOPERATOR}
U(i \tau) \colon L^2(M) \to \ocal(\partial M_\tau).
\end{equation}
Here, $\ocal(\partial M_\tau) = \ker(\bar{\partial}_b)$ denotes the space of CR holomorphic functions on the Grauert tube boundary. Comparing \eqref{eqn:WAVEOPERATOR} with \eqref{eqn:POISSONWAVEOPERATOR}, we find
\begin{equation}
U(i\tau, z, y) = \sum_{j = 0}^\infty e^{-\tau \lambda_j}  \phi_{\lambda_j}^\C(z) \overline{\phi_{\lambda_j}(y)}.
\end{equation}
Hence, the formula
\begin{equation}
\phi_{\lambda_j}^\C = e^{\tau \lambda_j} U(i\tau) \phi_{\lambda_j} \in \ocal(\pmt)
\end{equation}
defines the analytic extension of Laplace eigenfunctions \eqref{eqn:EIGENEQUATION}.

We define tempered spectral projections $P_\lambda$ by the eigenfunction partial sums
\begin{equation}
P_{\lambda} \colon \ocal(\partial M_\tau) \to \ocal(\partial M_\tau), \quad P_{\lambda}(z, w) = \sum_{j : \lambda_j \le \lambda} e^{-2\tau \lambda_j} \phi_{\lambda_j}^\C(z) \overline{\phi_{\lambda_j}^\C(w)}.
\end{equation}
The prefactor $e^{-2\tau \lambda_j}$ is introduced because of the exponential growth estimate (\cite[Corollary~3]{Zelditch12potential}) on complexified eigenfunctions
\begin{equation}
\lambda_j^{-\frac{m-1}{2}} e^{\tau \lambda_j} \lesssim \sup_{\zeta \in M_\tau} \abs{\phi_{\lambda_j}^\C(z)} \lesssim \lambda_j^{\frac{m-1}{2}} e^{\tau \lambda_j}.
\end{equation}

As before, fix a positive, even Schwartz function $\chi$ whose Fourier transform is compactly supported with $\hat{\chi}(0) = 1$. The techniques for proving \autoref{Scaling Theorem} are easily adapted to prove scaling asymptotics for the smoothed sum
\begin{align}
   \Pcl(z,w) = \chi \ast d_{\lambda} P_{\lambda}(z,w) = \sum_{j : \lambda_j \le \lambda} \chi( \lambda - \lambda_j) e^{ - 2 \tau \lambda_j}\phi^{\C}_{\lambda_j}(z) \overline{\phi^{\C}_{\lambda_j}(w)}. \label{Tempered Sum}
\end{align}

\begin{Th}[Scaling asymptotics for $P^\tau_{\lambda}$]\label{Scaling Theorem 2}
		Let $\Pcl$ be the spectral projection \eqref{Tempered Sum} associated to the Grauert tube boundary $\partial M_\tau$ of a closed, real analytic Riemannian manifold $M$ of dimension $m$. Fix \(p \in \pmt\). Then, in Heisenberg coordinates centered at $p = 0$,  the distribution kernel has the following scaling asymptotics:
		\begin{align}
\Pcl&\left( \frac{\theta}{\lambda}, \frac{u}{\sqrt{\lambda}}; \frac{\varphi}{\lambda}, \frac{v}{\sqrt{\lambda} }\right) \label{Smoothed temperate sums Scaled} \\
      &= \frac{C_{m,M}}{\tau^{m}}\lambda ^{\frac{m - 1}{2}} \Pi^{\mathbf{H}}_1\left( \frac{\theta}{2 \tau}, \frac{u}{\sqrt{\tau} }, \frac{\varphi}{2\tau}, \frac{v}{\sqrt{\tau} }  \right)\\
			&\quad \times \left(1 + \sum_{j = 1}^{N}\lambda^{ - \frac{j}{2}}P_j(p,u,v,\theta,\varphi) +  \lambda^{^{ - \frac{N + 1}{2}}} R_N\left(p, \theta, u , \varphi, v ,\lambda \right) \right) 
		\end{align}
        where \(\norm{R_{N}\left( p, \theta, u, \varphi, v, \lambda \right)}_{C^{j}\left( \{\abs{\theta} + \abs{\varphi} + |u| + |v| \leq \rho\}  \right)} \leq C_{K,j, \rho}\) for \(\rho > 0 , \ j = 1,2,3, \dots \) and \(P_j\) is a polynomial in \(u,v, \theta, \varphi\).      
\end{Th}

\begin{Rem}
    The leading order asymptotics of \autoref{Scaling Theorem} and \autoref{Scaling Theorem 2} differ by a factor of \(\lambda^{ - \frac{m - 1}{2}}\) because the analytic extensions \(\phi_{\lambda_j}^{\C}\) are not $L^2$-normalized on $\partial M_\tau$. 
\end{Rem}
\autoref{Scaling Theorem 2} is the complex analogue of \cite[Proposition~10]{CanzaniHanin2015}, where it was shown that the rescaled real spectral projection kernel $E_{\lambda}( x + \frac{u}{\lambda}, y + \frac{v}{\lambda})$ from \eqref{eqn:REALSPECTRALPROJ} exhibits Bessel-type scaling asymptotics as \(\lambda \to \infty\). The Heisenberg-scaling of the present paper is performed with parameter \(\sqrt{\lambda} \) rather than \(\lambda\). The complex geometry of the Grauert tube win over the real Riemannian geometry since as \(\tau \to 0\) our scaling asymptotics do not resemble Bessel asymptotics. It should also be mentioned that, unlike the authors of \cite{CanzaniHanin2015}, we do not study the sharp (i.e., without the smoothing $\chi$) spectral projection in this paper.

\subsection{Comparison to the line bundle setting}\label{sec:COMPARE}

Let $(L, h) \to (M, \omega)$ be a positive Hermitian line bundle over a closed K\"{a}hler manifold. The \emph{Bergman projections} are orthogonal projections
\begin{equation}\label{eqn:BERGMANLINE}
\Pi_{h^k} \colon L^2(M, L^k) \to H^0(M, L^k)
\end{equation}
from the space of $L^2$ sections of the $k$th tensor power of the line bundle onto the space of square integrable holomorphic sections.

The Bergman projections are related to the Szeg\H{o} projection
\begin{equation}\label{eqn:SZEGOLINE}
\Pi_h \colon L^2(\partial D) \to H^2(\partial D),
\end{equation}
where $\partial D = \set{\ell \in L^* : \norm{\ell}_{h^*} = 1}$ is the unit co-circle bundle. Indeed, let $r_\theta$ denote the circle action on $\partial D$, then the map
\begin{align}
H^0(M, L^k) &\to\{ f \in H^2(\partial D) : f(r_\theta x) = e^{ik\theta}f(x)\}&\\
s_k(z) &\mapsto f(x) = (\ell^{\otimes k}, s_k(z)) \quad \text{where $x = (\ell, z) \in \pi^{-1}(z) \times M$}
\end{align}
defines a unitary equivalence between the space of holomorphic sections of $L^k$ and the subspace of equivariant functions in $H^2(\partial D)$. Since the circle action on $\partial D$ commutes with $\overline{\partial}_b$, we conclude that the kernel of \eqref{eqn:BERGMANLINE} are Fourier coefficients of the kernel of \eqref{eqn:SZEGOLINE}, that is,
\begin{equation}\label{eqn:SPECTRALDECOMP}
\Pi_{h^k}(x,y) = \frac{1}{2\pi}\int_0^{2\pi} e^{-i k \theta} \Pi_h(r_\theta x,  y) \,d\theta = \p*{\frac{1}{2\pi}\int_0^{2\pi}e^{-ik\theta} r_\theta^*\Pi_h\,d\theta}(x,y).
\end{equation}
Note that the Fourier decomposition of $\Pi_h$ coincides with the spectral decomposition of $D_\theta = \frac{1}{i} \frac{\partial}{\partial \theta}$ on $\partial D$.

In the Grauert tube setting, $\partial M_\tau$ plays the role of the strongly pseudoconvex CR hypersurface $\partial D \subset L^*$ and the Hamilton vector field $\Xi_{\sqrt{\rho}}$ of the Grauert tube function on $\partial M_\tau$ plays the role of the Reeb vector field $\partial/\partial\theta$ on $\partial D$. With this analogy in mind, the direct analogue of \eqref{eqn:SPECTRALDECOMP} in the Grauert tube setting is
\begin{equation}\label{eqn:FIRSTGUESS}
\int_\R \hat{\chi}(t) e^{-i\lambda t} \Pi_\tau(\gtc{t}(z), w)\,dt = \p[\bigg]{\int_\R \hat{\chi}(t) e^{-i\lambda t}  (\gtc{t})^*\Pi_\tau \,dt}(z,w),
\end{equation}
in which
\begin{itemize}
\item integration against (the Fourier transform of) a suitably chosen Schwartz function $\hat{\chi}$ replaces integration on the circle in the line bundle setting;

\item pullback by the Hamilton flow \eqref{eqn:GTTAU} on the Grauert tube boundary replaces that of the circle action in the line bundle setting;

\item the orthogonal projection \eqref{eqn:SZEGOKERNEL} onto the Grauert tube boundary replaces the Szeg\H{o} projection in the line bundle setting.
\end{itemize}

It might appear that \eqref{eqn:FIRSTGUESS} is the correct object to study, but this is not quite right. Unlike the circle action on $\partial D$ generated by $\frac{\partial}{\partial \theta}$,  the geodesic flow on a Grauert tube is never holomorphic, nor does its orbits form a fiber bundle over a quotient space. In particular, $\gtc{t}$ need not commute with $\Pi_\tau$, so \eqref{eqn:FIRSTGUESS} fails to be CR holomorphic in the $z$ variable. This issue of holomorphy can be fixed: It is shown in \cite[Proposition~5.3]{Zelditch20Husimi} (see \autoref{Dynamical Toeplitz} for the statement) that there exists a polyhomogeneous pseudodifferential operator $\hat{\sigma}$ on $\partial M_\tau$ so that
\begin{equation}\label{eqn:DYNAMICTOEPLITZ}
\Pi_\tau e^{it\Pi_\tau D_{\sqrt{\rho}}\Pi_\tau} =  \Pi_\tau\hat{\sigma} (\gtc{t})^*\Pi_\tau \quad \text{modulo smoothing Toeplitz operators}.
\end{equation}
Evidently, the kernel of \eqref{eqn:DYNAMICTOEPLITZ} is CR holomorphic in both variables. Thus, we are led to the CR holomorphic analogue of \eqref{eqn:SPECTRALDECOMP} that is
\begin{equation}\label{eqn:SPECTRALDECOMPANALOGUE}
\int_\R \hat{\chi}(t) e^{-i\lambda t} \Pi_\tau \hat{\sigma}(\gtc{t})^*\Pi_\tau\,dt =  \int_\R \hat{\chi}(t) e^{-i\lambda t} \Pi_\tau e^{it\Pi_\tau D_{\sqrt{\rho}}\Pi_\tau}\,dt \quad \text{to leading order}.
\end{equation}

\begin{Rem}
We note that \eqref{eqn:SPECTRALDECOMPANALOGUE} is essentially the spectral decomposition of the elliptic Toeplitz operator $\Pi_\tau D_{\sqrt{\rho}}\Pi_\tau$ introduced in \eqref{eqn:TOEPLITZ}. An interesting problem is to compare the spectrum of the Toeplitz operator with that of $\sqrt{\Delta}$, as
 \begin{align}
 U({i\tau}) \sqrt{\Delta}^{\frac{m + 1}{2}}U({i\tau})^* =  \Pi_\tau D_{\sqrt{\rho} }\Pi_\tau \quad\text{modulo zeroth order Toeplitz operator}
\end{align}
where $U(i\tau)$ is the Poisson-wave operator discussed in \autoref{sec:POISSONWAVE}. We include further discussions in this direction in a subsequent paper  \cite[Section~5]{ChangRabinowitz21}.
\end{Rem}

\subsection{Related results in the line bundle setting}\label{sec:RELATED}

Scaling asymptotics in the line bundle setting have been proved in varying degrees of generality; see \cite{BleherShiffmanZelditch00, ShiffmanZelditch02, MaMarinescu07, MaMarinescu13, LuShiffman15, HezariKelleherSetoXu16}. The simplest setup is a closed K\"{a}hler manifold $M$ polarized by an ample line bundle $L$. Endow $L$ with a Hermitian metric $h$ so that the curvature two-form $\Omega$ is positive, and let $\omega = \frac{1}{2}\Omega$ be the K\"{a}hler form on $M$.

Choose coordinates $z = (z_1, \dotsc, z_m)$ in a neighborhood $U$ centered at $p \in M$ so that the K\"{a}hler potential is locally of the form $\phi(z) = \abs{z}^2 + O(\abs{z}^3)$. We identify the Bergman kernel with the Fourier components of the Szeg\H{o} kernel using \eqref{eqn:SZEGOLINE} and write
\begin{equation}
\Pi_{h^k}(x,y) = \Pi_{h^k}(z, \theta, w, \phi),
\end{equation}
where $x = (z, \theta)$ are linear coordinates on $T_zM \times S^1$. Then the Bergman kernel scaling asymptotics is of the form
\begin{equation}\label{eqn:LINEBUNDLE}
\begin{aligned}
k^{-m} \Pi_{h^k}\p[\bigg]{\frac{z}{\sqrt{k}}, \frac{\theta}{k}, \frac{w}{\sqrt{k}}, \frac{\phi}{k}} &= \frac{1}{\pi^m} e^{i(\theta - \phi) + i \Im(z \cdot \bar{w})- \frac{1}{2}\abs{z - w}^2}\\
&\quad \times \bra[\bigg]{1 + \sum_{n = 1}^N k^{-\frac{n}{2}}b_k(z, w) + k^{-\frac{N+1}{2}} R_N(z, w)}
\end{aligned}
\end{equation}
with
\begin{equation}
\norm{R_N}_{C^j(\{ \abs{z} + \abs{w} \le \delta\})} \le C(N, j, \delta).
\end{equation}
The remainder estimate implies Gaussian decay of the Bergman kernel in a neighborhood of the diagonal:
\begin{equation}\label{eqn:GAUSSIANDECAY}
\abs[\Big]{\Pi_{h^k}\p{z ,\theta, w, \phi}} \le \frac{k^m}{\pi^m} e^\frac{-kD(z,w)}{2}\p[\bigg]{1 + O\p[\bigg]{\frac{1}{k}}} \quad \text{whenever $d(z,w) \lesssim \sqrt{\frac{\log k}{k}}$}.
\end{equation}
Here, $D(z,w)$ is the Calabi diastasis \eqref{eqn:DIASTASIS}, which is bounded above and below by the square $d^2(z,w)$ of the Riemannian distance function. 

A complete asymptotic expansion far away from the diagonal is available thanks to the recent papers \cite{RoubySjostrandNgoc20, Delporte21, HezariXu20,DeleporteHitrikSjostrand21} on \emph{analytic Bergman kernel}. For instance, \cite[Theorem~1.1]{HezariLuXu18} shows (under analyticity assumption of the K\"{a}hler potential) that
\begin{equation}
\abs[\Big]{\Pi_{h^k}\p{z ,\theta, w, \phi}} = \frac{k^m}{\pi^m} e^\frac{-kD(z,w)}{2}\p[\bigg]{1 + O\p[\bigg]{\frac{1}{k}}} \quad \text{whenever $d(z,w) \lesssim \frac{1}{k}$}.
\end{equation}
In more general settings (including when the K\"{a}hler potential is no longer assumed to be analytic), global Agmon-type estimates
\begin{equation}
\abs[\Big]{\Pi_{h^k}\p{z ,\theta, w, \phi}} \le C e^{-c \sqrt{k}d(z,w)} \quad \text{for all $z, w \in M$}
\end{equation}
have been established in \cite{Christ91, Delin98, Lindholm01, Berndtsson03, MaMarinescu15, Zelditch16Toric, Christ18}. 

In the presence of a real Hamiltonian \(f \in C^{\infty}( M )\), there is a Hamiltonian flow \(\Phi^{M}_{\tau}\) and its lift \(\Phi^{L}_{\tau}\) to \(L\) in addition to the circle action. In \cite{Paoletti20121}, the phase space scaling asymptotics is computed near a fixed point \(m\) and time \(\tau\) in the graph of \(\Phi_{\tau}^{M}\) for the dynamical Toeplitz operators  
\begin{align}
    U_{\tau} = R_{\tau} \circ \Pi \circ ( \Phi^{L}_{ -\tau} \circ \Pi ),
\end{align}
where \(R_{\tau}\) is a \(0\)th order Toeplitz operator. These asymptotics have an additional dependence on the derivative of \(\Phi^{M}_{ -\tau}\) at \(m\). The nature of the computation in our proof is similar to the one in \cite{Paoletti20121}.

\section*{Acknowledgment}
\noindent The authors would like to thank their thesis advisor Steve Zelditch for many helpful discussions and guiding their attention towards Szeg\H{o} kernels and Grauert tubes. The referee, whose careful reading of the manuscript led to various corrections and improvement, is also much appreciated.

\section{Geometry of and analysis on Grauert tubes}

We assume throughout that \((M,g)\) is a closed, real analytic Riemannian manifold of dimension $m$. Grauert tubes were first studied by \cite{GuilleminStenzel91,GuilleminStenzel92,LempertSzoke91,LempertSzoke01} in relation to the complex Monge--Amp\`{e}re equation and complexified geodesics. The Poisson wave operator and the tempered spectral projection \eqref{eqn:TEMPEREDPROJ}, as well as other related operators, have been studied in \cite{Zelditch07complex,Zelditch12potential,Zelditch20Husimi}.

\subsection{Grauert tube and the cotangent bundle}\label{sec:GTUBE}

Bruhat--Whitney proved that a real analytic manifold $M$ admits a complexification \(M_{\C}\) such that \(M \subset M_{\C}\) is a totally real submanifold, that is, $T_pM \cap JT_pM = \{0\}$ for all $p \in M$.

In a neighborhood of $M$ in $M_\C$, there exists a unique strictly plurisubharmonic function \(\rho \colon U \subset M_\C \to [0,\infty)\) whose square root is the solution of the Monge--Amp\`{e}re equation
    \begin{equation}
        \det\p[\bigg]{\frac{\partial^2 \sqrt{\rho}}{\partial z_j \partial \overline{z}_k}} = 0 \quad \text{on $U \setminus M$},
    \end{equation}
    with the initial condition that the metric induced by the K\"{a}hler form \(i \partial \overline{\partial} \rho \) restricts to the Riemannian metric \(g\) on \(M\). In fact, $\sqrt{\rho}$ is given by
\begin{equation}\label{eqn:GRAUERTTUBEFUNCTION}
    \sqrt{\rho} \colon U \subset M_\C \to \R, \quad \sqrt{\rho}(z) = \frac{1}{2i} \sqrt{r^2_\C(z, \overline{z})},
\end{equation}
where  \(r^2_\C(z,\overline{w})\) is the analytic extension of the square of the Riemannian distance function \(r \colon M \times M \to \R\) to a neighborhood of the diagonal in $M_\C \times \overline{M}_\C$. We call \eqref{eqn:GRAUERTTUBEFUNCTION} the \emph{Grauert tube function}. For each \(0 < \tau < \tau_{\mathrm{max}}\), the sublevel set
\begin{equation}\label{eqn:MTAU}
    M_{\tau} = \{z \in M_\C : \sqrt{\rho}(z) < \tau\}
\end{equation}
is called the \emph{Grauert tube of radius \(\tau\)}. 

The complexified exponential map can be used to identify \eqref{eqn:MTAU} with the co-ball bundle of radius $\tau$:
\begin{equation}
B^*_\tau M  = \{(x,\xi): \abs{\xi} < \tau\}.
\end{equation}
To explain this, we introduce notation for geodesic flows.
\begin{itemize}
\item The (geometer's) geodesic flow $g^t \colon T^*M \to T^*M$ is the Hamilton flow of the metric norm squared function $H(x,\xi) = \abs{\xi}^2_x$ . 

\item The homogeneous geodesic flow $G^t \colon T^*M - 0 \to T^*M - 0$ is the Hamilton flow of the metric norm function $\sqrt{H}(x,\xi) = \abs{\xi}_x$. Note that $G^t(x, \epsilon \xi) = \epsilon G^t(x,\xi)$.
\end{itemize}

Let $\pi \colon T^*M \to M$ be the natural projection. The exponential map $\exp \colon T^*M \to M$ is defined by $\exp_x \xi = \pi g^1(x,\xi)$. Analyticity of the metric allows us to complexify \(t \mapsto t + i \tau\) in the time variable. For $0 < \tau < \tau_{\mathrm{max}}$, the imaginary time geodesic flow $E$ is a diffeomorphism (\cite[Lemma~1.1]{Zelditch14nodalgeo})
\begin{equation}\label{eqn:DIFFEO}
    E \colon B^{*}_{\tau}M \to M_\tau, \quad E(x,\xi) = \exp_x^\C i\xi.
\end{equation}
With $\sqrt{H}(x,\xi) = \abs{\xi}$ and $\sqrt{\rho}$ the Grauert tube function, we have
\begin{equation}
H = \rho \circ E \quad \text{and} \quad \sqrt{H} = \sqrt{\rho} \circ E.
\end{equation}
It follows that $E^{-1}$ conjugates the geodesic flow $G^t$ (resp.~$g^t$) on the cotangent bundle to the Hamilton flow of Hamiltonian vector field $\Xi_{\sqrt{\rho}}$ (resp.~$\Xi_\rho$) on the Grauert tube. In particular, the transfer
\begin{equation}\label{eqn:GTTAU}
G^t_\tau \colon \partial M_\tau \to \partial M_\tau, \quad G^t_\tau = E \circ G^t \circ E^{-1}|_{\partial M_\tau}
\end{equation}
of the geodesic flow from the energy surface $\partial B_\tau^*M$ to the Grauert tube boundary $\partial M_\tau$ coincides with the restriction of the Hamilton flow of $\Xi_{\sqrt{\rho}}$ to $\partial M_\tau$.

Finally, if we write $\alpha_{T^*M} = \xi \,dx$ and $\omega_{T^*M} = d\xi \wedge dx$ for the canonical 1-form and the symplectic form on the cotangent bundle, then we also have
\begin{equation}\label{eqn:EQUAL}
(E^{-1})^* \alpha_{T^*M} = \Imp \partial  \rho = d^c\sqrt{\rho} \quad \text{and} \quad (E^{-1})^*\omega_{T^*M} = -i \partial \bar{\partial} \rho.
\end{equation}
Note that the right-hand sides of the two equalities are the canonical 1-form and the K\"{a}hler form on the Grauert tube $M_\tau$.

\begin{Rem}\label{rem:ADAPTEDJ}
It is useful to think of the Grauert tube as the co-ball bundle endowed with an \emph{adapted complex structure} $J = J_g$ induced from the Riemannian metric on $M$. This complex structure is characterized as follows. Let $\gamma \colon \R \to M$ be a geodesic (i.e., a $g^t$ orbit) and let $\gamma_\C \colon \set{t + i \tau \in \C: \tau < \tau_{\mathrm{max}}} \to M_\C$ be its analytic continuation to a strip. Then $J_g$ is the complex structure on $M_\tau = B^{*}_{\tau}M$ so that $\gamma_\C$ is a holomorphic map; see \cite{LempertSzoke91}.
\end{Rem}

\subsection{Contact and CR structure on the boundary of a Grauert tube}\label{sec:CONTACT} 

The boundary of a Grauert tube, being a level set of the strictly plurisubharmonic function $\rho$, is strongly pseudoconvex and real analytic. We endow $\partial M_\tau$ with the volume form $d\mu_\tau$ obtained by pulling back the standard Liouville form $\alpha \wedge \omega_{T^*M}^{m-1}$ on $\partial B_\tau^*M_\tau$ under the symplectic diffeomorphism \eqref{eqn:DIFFEO}:

\begin{equation}\label{eqn:CONTACTVOLUME}
d \mu_{\tau} = (E^{-1})^*(\alpha \wedge \omega_{T^*M}^{m-1})\Big|_{\partial M_\tau}. 
\end{equation}

Let $J = J_g$ denote the adapted complex structure on a Grauert tube (see \autoref{rem:ADAPTEDJ}). The boundary $\partial M_\tau$, being a real hypersurface of the complex manifold $M_\tau$, carries a CR structure. The tangent space admits the decomposition
\begin{equation}
T\partial M_\tau = H \oplus \R T,
\end{equation}
where $H = \ker \alpha = JT\partial M_\tau \cap T\partial M_\tau$ is a real, $J$-invariant hyperplane bundle and \(T  = \Xi_{\sqrt{\rho} }\), called the \emph{Reeb vector field}, is the Hamilton vector field of the Grauert tube function. Equivalently, $JT = \nabla \sqrt{\rho}$.

Complexifying the tangent space yields \begin{align}
    T^{\C} \pmt = H^{(1,0)} \pmt \oplus H^{(0,1)} \pmt \oplus \C T
\end{align}
where \(H^{(1,0)}\) and \(H^{(0,1)}\) are the $J$-holomorphic and $J$-antiholomorphic subspaces. The boundary Cauchy--Riemann operators are defined by \begin{align}
    \partial_{b}f = df|_{H^{( 1,0 )}} \quad \text{and} \quad \overline{\partial}_{b}f = df |_{H^{( 0,1 )}}.
\end{align}

\subsection{The Szeg\H{o} projector and the Boutet de Movel--Sj\"{o}strand parametrix}

The Szeg\H{o} projector is a complex Fourier integral operator (FIO) with a positive complex canonical relation. We recall the Boutet de Monvel--Sj\"{o}strand parametrix construction for the Szeg\H{o} kernel in the context of Grauert tubes, but the same construction holds with $\partial M_\tau$ replaced by any bounded, strongly pseudoconvex domain.

 \begin{Def}
     The Szeg\H{o} projector \(\Pi_{\tau}\) is the orthogonal projection
     \begin{align}\label{eqn:SZEGOPROJ}
         \Pi_\tau \colon L^{2}(\pmt, d \mu_{\tau}) \to H^{2}( \pmt, d \mu_{\tau})
     \end{align}
     onto the Hardy space consisting of boundary values of holomorphic functions in $M_\tau$ that are square integrable with respect to the volume form \eqref{eqn:CONTACTVOLUME}. Its distributional kernel \(\Pi_\tau(z,w)\) is defined by the relation
     \begin{align}\label{eqn:SZEGOKERNEL}
         \Pi_\tau f(z) = \int_{\pmt} \Pi_\tau(z,w)f(w) \, d \mu_{\tau}(w) \quad  \text{for all \(f \in L^{2}(\pmt)\)}.
     \end{align}
 \end{Def}

The Szeg\H{o} projector is a Fourier integral operator with a positive complex canonical relation. The real points of this canonical relation is the graph of the identity map on the symplectic cone $\Sigma_\tau$ spanned by the contact form $\alpha := d^c\sqrt{\rho}|_{T*\partial M_\tau}$, i.e.,
\begin{equation}\label{eqn:SIGMATAU}
\Sigma_\tau = \set{(\zeta, r \alpha_\zeta): r \in \R_+} \subset T^*(\partial M_\tau).
\end{equation}
Using the imaginary time geodesic flow $E$ in \eqref{eqn:DIFFEO}, we can construct the symplectic equivalence $\iota_\tau$ between the cotangent bundle and the symplectic cone:
\begin{equation}\label{eqn:IOTA}
\iota_\tau \colon T^*M - 0 \to \Sigma_\tau, \quad \iota_\tau(x, \xi) = \p[\bigg]{E\p[\Big]{x, \tau \frac{\xi}{\abs{\xi}}}, \abs{\xi} \alpha_{E(x, \tau \frac{\xi}{\abs{\xi}})}}.
\end{equation}

We now briefly describe the symbol of the Szeg\H{o} projector. Details can be found in \cite[Theorem~11.2]{BoutetGuillemin81} or \cite[Section~3]{Zelditch96cstar}. Let $\Sigma_\tau^\perp \otimes \C$ be the complexified normal bundle of \eqref{eqn:SIGMATAU}. The symbol $\sigma(\Pi_\tau)$ of $\Pi_\tau$ is a rank one projection onto a ground state $e_{\Lambda_\tau}$, which is annihilated by a Lagrangian system of Cauchy--Riemann equations corresponding to a Lagrangian subspace $\Lambda_\tau \subset \Sigma_\tau^\perp \otimes \C$. The time evolution $\Pi_\tau \mapsto \gtc{-t}\Pi_\tau\gtc{t}$ under the Hamilton flow \eqref{eqn:GTTAU} yields another rank one projection onto some time-dependent ground state $e_{\Lambda_\tau^t}$, where $\Lambda_\tau^t$ is the pushforward of $\Lambda_\tau$ under the flow. The quantity
\begin{equation}\label{eqn:OVERLAP}
\sigma_{t, \tau, 0} = \ang{e_{\Lambda_\tau^t} , e_{\Lambda_\tau}}^{-1}
\end{equation}
appears in \autoref{Dynamical Toeplitz Wave group} and \autoref{Dynamical Toeplitz}. It was shown in \cite{Zelditch96cstar} that $\ang{e_{\Lambda_\tau^t} , e_{\Lambda_\tau}}$ is the $L^2$ inner product of two Gaussians, hence nowhere vanishing.

Finally, we discuss an oscillatory integral representation of the Szeg\H{o} kernel. Recall from \autoref{sec:GTUBE} that $\rho$ is the real analytic K\"{a}hler potential on the Grauert tube. We introduce the defining function for the boundary $\partial M_\tau$ of a Grauert tube:
\begin{equation}\label{eqn:DEFININGFUNCTION}
\ptau \colon M_{\tau_\mathrm{max}} \to [0,\infty), \quad \ptau(z) := \rho(z) - \tau^2
\end{equation}
Let $\ptau(z,\overline{w})$ be the analytic extension of $\ptau(z) = \ptau(z,\overline{z})$ to $M_\tau \times \overline{M}_\tau$ obtained by polarization. Put (c.f.~the formula \eqref{eqn:GRAUERTTUBEFUNCTION} for the Grauert tube function)
\begin{align}\label{eqn:BSJPHASE}
    \psi_\tau(z,w) = \frac{1}{i} \ptau(z, \overline{w}) = \frac{1}{i}\left( - \frac{1}{4}r^{2}_{\C}(z, \overline{w} ) - \tau^{2} \right).
\end{align}
By construction, $\psi$ is homolomorphic in $z$, antiholomorphic in $w$, and satisfies $\psi(z,w) = -\overline{\psi(z,w)}$. 

\begin{Th}[The Boutet de Monvel--Sj\"{o}strand parametrix, {\cite[Theorem~1.5]{BoutetSjostrand76}}]
     With $\psi$ as in \eqref{eqn:BSJPHASE}, there exists a classical symbol
\begin{equation}
s \in S^{m - 1 }( \pmt \times \pmt \times \R^{ +} ) \quad \text{with} \quad s(z,w, \sigma) \sim \sum_{k = 0}^{\infty}\sigma^{m - 1 - k}s_k(z,w) \label{Symbol Expansion BDM}
\end{equation}
so that the Szeg\H{o} kernel \eqref{eqn:SZEGOKERNEL} has the oscillatory integral representation
    \begin{align}
        \Pi_\tau(z,w) = \int_{0}^{\infty}e^{i \sigma \psi_\tau(z,w)}s(z, w,\sigma) \, d \sigma  \quad \text{modulo a smoothing kernel}. \label{Parametrix for Szego Kernel}
    \end{align}
\end{Th}

We record a key estimate for $\ptau$ (or equivalently for the phase function $\psi$). Define the \emph{Calabi diastatis function} by
\begin{equation}\label{eqn:DIASTASIS}
D(z, w) = \ptau(z,\overline{z}) + \ptau(w, \overline{w}) - \ptau(z, \overline{w}) - \ptau(w, \overline{z}).
\end{equation}
On \(\pmt\), the diastatis simplifies to
\begin{align}\label{eqn:DIASTASISPMT}
    D(z,w) = -\ptau(z, \overline{w}) - \ptau(w, \overline{z}) = - 2 \Rep \ptau(z,\overline{w}) \quad \text{for $z, w \in \partial M_\tau$}.
\end{align}
In the closure of the Grauert tube, \cite[Corollary~1.3]{BoutetSjostrand76} gives the lower bound
\begin{align}\label{Diastasis inequality}
    D(z,w) \geq C \p[\big]{ d( z, \pmt ) + d( w, \pmt) + d(z,w)^{2}} \quad \text{for $z,w \in \overline{M_{\tau}}$}.
\end{align}

\subsection{Poisson wave operator and tempered spectral projection}\label{sec:POISSONWAVE}
Recall the eigenequation \eqref{eqn:EIGENEQUATION} for the Laplacian on $M$. The eigenfunction expansion of the Schwartz kernel of the half-wave operator $U(t) := e^{it\sqrt{\Delta}}$ is given by
\begin{align}\label{eqn:WAVE}
     U(t,x,y) = \sum_{j=0}^\infty e^{i t \lambda_j} \phi_{\lambda_j}(x) \overline{\phi_{\lambda_j}(y)}.
\end{align}
It is well-known (see for instance \cite[Theorem]{GuilleminStenzel92}) that for $0 \le \tau \le \tau_{\mathrm{max}}$, the Schwartz kernel $U(t,x,y)$ admits an analytic extension $U(t + i \tau,x,y)$ in the time variable $t \mapsto t + i\tau \in \C$. Note that the corresponding operator $U(i\tau) = e^{-\tau \sqrt{\Delta}}$ is the Poisson operator. Moreover, for $y$ and $\tau$ fixed, the Poisson kernel can be further analytically extended in the spacial variable $x \mapsto z \in M_\tau$. The resulting operator is a Fourier integral operator of complex type. More precisely, denote by $\ocal(\partial M_\tau)$ the space of CR holomorphic functions on the Grauert tube boundary.

\begin{Th}[{\cite{Boutet78, GuilleminStenzel92,GolseLeichtnamStenzel96,Zelditch12potential}}]\label{theo:POISSON}
For $0 < \tau < \tau_{\mathrm{max}}$, the Poisson operator
\begin{equation}
U(i \tau) = e^{-\tau \sqrt{\Delta}} \colon L^2(M) \to \ocal(\partial M_\tau)
\end{equation}
whose Schwartz kernel  $U(i\tau, z, y)$ is obtained from analytically continuing the half-wave kernel $U(t,x,y)$ of \eqref{eqn:WAVE} is a Fourier integral operator of order $-(m-1)/4$ with complex phase associated to the canonical relation
\begin{equation}
\set{(y, \eta, \iota_\tau(y, \eta)} \subset T^*M \times \Sigma_\tau,
\end{equation}
where $\iota_\tau$ and $\Sigma_\tau$ are defined in \eqref{eqn:IOTA} and \eqref{eqn:SIGMATAU}.
\end{Th}

To introduce the complexified spectral projection kernels, we need to further continue $U(i\tau, z,y)$ anti-holomorphically in the $y$ variable. Consider the operator
\begin{equation}
U^\C(t + 2i \tau) = U(i\tau) U(t) U(i\tau)^* \colon \ocal(\partial M_\tau) \to \ocal(\partial M_\tau)
\end{equation}
with Schwartz kernel
\begin{equation}
U^\C(t + 2i\tau, z, w) = \sum_{j=0}^\infty e^{i ( t + 2 i \tau ) j}\phi^{\C}_{\lambda_j}(z) \overline{\phi^{\C}_{\lambda_j}(w)}.
\end{equation}

\begin{Prop}[{\cite[Proposition~7.1]{Zelditch20Husimi}}] \label{Dynamical Toeplitz Wave group}
    Let $(\gtc{t})^*$	denote the pullback by the Hamilton flow \eqref{eqn:GTTAU} of $\Xi_{\sqrt{\rho}}$ on $\partial M_\tau$. There exists a classical polyhomogeneous pseudodifferential operator \(\hat{\sigma}_{t, \tau}( w, D_{\sqrt{\rho} })\) on \(\pmt\) 
    so that
    \begin{align}\label{Wave Group Toeplitz}
        U^{\C}( t + 2 i \tau ) = \Pi_\tau \hat{\sigma}_{t, \tau} \left( \gtc{t} \right)^{*} \Pi_\tau \quad \text{modulo a smoothing Toeplitz operator}. 
    \end{align}
The symbol $\sigma_{t,\tau}$ of $\hat{\sigma}_{t,\tau}$ admits a complete asymptotic expansion
    \begin{align}\label{eqn:WAVEGROUPTOEPLITZ}
        \sigma_{t, \tau}\left( w,r \right) \sim \sum_{j = 0}^{\infty}\sigma_{t, \tau, j}\left( w \right)r^{ - \frac{m - 1}{2} - j},
    \end{align}
in which $\sigma_{t,\tau, 0} =  \ang{e_{\Lambda_\tau^t} , e_{\Lambda_\tau}}^{-1}$ is the reciprocal of the overlap of two Gaussians, as in \eqref{eqn:OVERLAP}.
\end{Prop}

Finally, we define the tempered spectral projection kernel by
\begin{equation}\label{eqn:TEMPEREDPROJ}
P_\lambda(z,w) = \sum_{j: \lambda_j \le \lambda} e^{-2\tau \lambda_j}\phi^{\C}_{\lambda_j}(z) \overline{\phi^{\C}_{\lambda_j}(w)}.
\end{equation}
Fix \(\epsilon > 0\) and let \(\chi\) be a positive even Schwartz function such that \(\hat{\chi} \left( 0 \right) = 1 \) and \( \supp\hat{ \chi} \subset [ - \epsilon, \epsilon]\). Then
\begin{align}\label{eqn:TEMPEREDPCL}
    \Pcl(z,w) = \chi \ast d_{\lambda} P_\lambda( z,w) =  \int_{\R}\hat{\chi} (t) e^{ - it \lambda} U^{\C}( t + 2 i \tau, z , w) \, dt
\end{align}

\begin{Rem}
It was shown in \cite{Zelditch12potential} that the complexified spectral function
\begin{equation}
\sum_{\lambda_j \le \lambda} \abs{\phi_{\lambda_j}^\C(\zeta)}^2 \qquad (\zeta \in M_\tau)
\end{equation}
grows exponentially at a rate of $e^{2\lambda \sqrt{\rho}(\zeta)}$. We introduce the exponentially damping prefactor to obtain polynomial growth in the tempered version \eqref{eqn:TEMPEREDPROJ}.
\end{Rem}

\subsection{The Toeplitz operator \texorpdfstring{$\Pi_\tau D_{\sqrt{\rho}}\Pi_\tau$}{PDP}}
Let \(D_{\sqrt{\rho} } = \frac{1}{i} \Xi_{\sqrt{\rho} }\) denote the Hamilton vector field of the Grauert tube function acting as a differential operator. The symbol of \(D_{\sqrt{\rho} }\) is nowhere vanishing on $\Sigma_\tau - 0$, where $\Sigma_\tau$ is the symplectic cone \eqref{eqn:SIGMATAU}. Thus, \(\Pi_\tau D_{\sqrt{\rho}}\Pi_\tau\) is an elliptic Toeplitz operator with discrete spectrum.

As discussed in \autoref{sec:COMPARE}, the Grauert tube analogue of Fourier coefficients of the Szeg\H{o} kernel are given by the spectral localization operator 
\begin{align}
    \pcl = \Pi_\tau \chi ( \Pi_\tau D_{\sqrt{\rho}} \Pi_\tau - \lambda )  = \int_{\R} \hat{\chi} (t) e^{ - it \lambda}\Pi_\tau e^{i t \Pi_\tau D_{\sqrt{\rho}}\Pi_\tau}  \, dt
\end{align}

\begin{Prop}[{\cite[Proposition~5.3]{Zelditch20Husimi}}] \label{Dynamical Toeplitz}
    Let $(\gtc{t})^*$	denote the pullback by the Hamilton flow \eqref{eqn:GTTAU} of $\Xi_{\sqrt{\rho}}$ on $\partial M_\tau$. There exists a classical polyhomogeneous pseudodifferential operator \(\hat{\sigma}_{t, \tau}( w, D_{\sqrt{\rho} })\) on \(\pmt\) 
    so that
    \begin{align}\label{PDP Toeplitz}
        \Pi_\tau e^{i t \pdp } = \Pi_\tau\hat{\sigma}_{t, \tau} \left( \gtc{t} \right)^{*} \Pi_\tau \ \quad \text{modulo a smoothing Toeplitz operator}. 
    \end{align}
The symbol $\sigma_{t,\tau}$ of $\hat{\sigma}_{t,\tau}$ admits a complete asymptotic expansion
    \begin{align}\label{eqn:PDPTOEPLITZ}
        \sigma_{t, \tau}(w,r) \sim \sum_{j = 0}^{\infty}\sigma_{t, \tau, j}(w)r^{- j},
    \end{align}
in which $\sigma_{t,\tau, 0} =  \ang{e_{\Lambda_\tau^t} , e_{\Lambda_\tau}}^{-1}$ is the reciprocal of the overlap of two Gaussians, as in \eqref{eqn:OVERLAP}.
\end{Prop}

\begin{Rem}\label{rem:ORDER}
    The difference between \autoref{Dynamical Toeplitz} and \autoref{Dynamical Toeplitz Wave group} is that the symbol \eqref{eqn:WAVEGROUPTOEPLITZ} is of order zero and \eqref{eqn:PDPTOEPLITZ} is of order \( - (m - 1)/2\). This is because the partial sums \eqref{eqn:TEMPEREDPROJ} are not divided through by \( \norm{e^{ - \tau \lambda}\phi_{\lambda}}_{L^{2}( \pmt )}^2 \sim \lambda^{ -\frac{m - 1}{2}}\). For an in-depth discussion of complexified eigenfunctions and dynamical Toeplitz operators see \cite{Lebeau,Zelditch12potential,Zelditch20Husimi}.
\end{Rem}

\section{Heisenberg coordinates on the boundary of a Grauert tube}\label{sec:HEISENBERG}
This section reviews the construction and properties of Heisenberg coordinates following \cite[Section~18]{FollandStein}. Roughly speaking, in these coordinates, the strongly pseudoconvex boundary $\partial M_\tau \subset M_\C$ is, to a first approximation, the Heisenberg group viewed as a hypersurface in complex Euclidean space. We briefly recall the CR structure on the Heisenberg group.

The Heisenberg group $\mathbf{H}^{m-1}$ of degree $m-1$ is the Lie group $\R \times \C^{m-1}$ with group law
\begin{equation}\label{eqn:GROUPLAW1}
(\theta, z) \cdot (\varphi, w) = (\theta + \varphi + 2 \Imp (z \cdot w), z + w).
\end{equation}
\begin{Rem}
It is more common to use $t$ for the $\R$-coordinate but, since $t$ is already reserved for the time parameter of the geodesic or Hamiltonian flow, we use $\theta$ instead.
\end{Rem}
For $1 \le j \le m-1$, the vector fields
\begin{equation}
Z_j = \frac{\partial}{\partial z_j} + i \overline{z}_j \frac{\partial}{\partial \theta}, \qquad  \overline{Z}_j = \frac{\partial}{\partial \overline{z}_j} + i z_j \frac{\partial}{\partial \theta}, \qquad  T = \frac{\partial}{\partial \theta}
\end{equation}
satisfy the commutation relations
\begin{equation}
\bra{Z_j, \overline{Z}_k} = -2i \delta_{jk} T \quad \text{and} \quad \bra{Z_j, Z_k} = \bra{\overline{Z}_j, \overline{Z}_k} = \bra{Z_j, T} = \bra{\overline{Z}_j, T} = 0.
\end{equation}

Furthermore, the $1$-form
\begin{equation}
\alpha = d \theta - i \sum_{j=1}^{m-1} (\overline{z}_j\,dz_j - z_j \,d\overline{z}_j)
\end{equation}
annihilates $Z_j$ and $\overline{Z}_j$ for all $1 \le j \le m-1$. Hence, the complexified tangent space of the Heisenberg group admits the decomposition
\begin{equation}
T^\C \mathbf{H}^{m-1} = H^{(1,0)}\oplus H^{(0,1)} \oplus \C T,
\end{equation}
in which $H^{(1,0)}$ and $H^{(0,1)}$ are spanned by the holomorphic vector fields $Z_j$ and the anti-holomorphic vector fields $\overline{Z}_j$, respectively. Direct verification shows that the subspace $H^{(1,0)}$ defines a CR structure on $\mathbf{H}^{m-1}$.

The Levi form $\ang{\, , \, }_L$ is the Hermitian form on $H^{(1.0)}$ defined by
\begin{equation}
\ang{Z, W}_L = -i \ang{d\alpha, Z \wedge \overline{W}} \quad \text{for all $Z, W \in H^{(1,0)}$}.
\end{equation}
Direct computation shows that the Levi form with respect to the basis $Z_j$ is the identity. Thus, $\mathbf{H}^{m-1}$ is strongly pseudoconvex.

Set
\begin{align}
D_{m} &= \set[\bigg]{\zeta = (\zeta_0, \dotsc, \zeta_{m-1}) \in \C^m : \sum_{j=1}^{m-1} \abs{\zeta_j}^2 < \Imp \zeta_0}.\\
\partial D_m &= \set[\bigg]{\zeta = (\zeta_0, \dotsc, \zeta_{m-1}) \in \C^m :  \sum_{j=1}^{m-1} \abs{\zeta_j}^2  = \Imp \zeta_0}. \label{eqn:MODELHEISENBERG}
\end{align}
The Heisenberg group $\mathbf{H}^{m-1}$ acts on $\C^m$ by holomorphic affine transformations; given $(\theta, z) = (\theta, z_1, \dotsc, z_{m-1}) \in \mathbf{H}^{m-1}$ and $\zeta = (\zeta_0, \dotsc, \zeta_{m-1}) \in \C^m$, we have $(\theta, z) \cdot \zeta = (\zeta_0', \dotsc, \zeta_{m-1}')$, where
\begin{equation}\label{eqn:GROUPLAW2}
\begin{aligned}
\zeta_0' &= \zeta_0 +  \theta + i \abs{z}^2 + 2i \sum_{j=1}^{m-1} \zeta_j \overline{z}_j,\\
\zeta_j' &= \zeta_j + z_j  \quad \text{for  $1 \le j \le m-1$}.
\end{aligned}
\end{equation}
 Note that the group action preserves $D_m$ (which is biholomorphic to the unit ball in $\C^m$) and $\partial D_m$. Indeed, $\mathbf{H}^{m-1}$ acts simply and transitively on $\partial D_m$. Thus, comparing \eqref{eqn:GROUPLAW1} and \eqref{eqn:GROUPLAW2}, we find that the Heisenberg group may be identified to $\partial D_m$ via the correspondence
\begin{equation}
H_{m-1} \ni (\theta, z) \leftrightarrow (\theta, z) \cdot 0 = ( \theta + i \abs{z}^2, z_1, \dotsc, z_{m-1}) \in \partial D_m.
\end{equation}


On any nondegenerate CR manifold (in particular the boundary of a Grauert tube), there exists a Heisenberg-like coordinate system $( \theta, z_1, \dotsc, z_{m-1})$ so that $Z_j = (\partial /\partial z_j) + i \overline{z}_j (\partial / \partial \theta)$ and $T = (\partial / \partial \theta)$ up to lower order terms. The precise statement, \autoref{defn:COORDINATES}, is already given in the introduction, so we do not reproduce it here. An explicit construction of such ``normal coordinates'' was presented in \cite{FollandStein}. We call the same coordinates ``Heisenberg coordinates'' in this paper to emphasize the osculating Heisenberg structure (see \autoref{rem:HEISENBERG}).

\begin{Rem}
If $f$ is smooth, then
\begin{align}
f &= O^1 \text{ if and only if $f(\eta) = O\p*{\sum_{k=1}^{m-1} (\abs{x_k(\eta)} + \abs{y_k(\eta)}) + \abs{\theta(\eta)}}$},\\
f &= O^2 \text{ if and only if $f(\eta) = O\p*{\sum_{k=1}^{m-1} (\abs{x_k(\eta)}^2 + \abs{y_k(\eta)}^2) + \abs{\theta(\eta)}}$}.
\end{align}
\end{Rem}

\subsection{Taylor expansions in Heisenberg coordinates}

We fix notation. Let $p \in \partial M_\tau$ and let $z ,w$ be two points in $M_{\tau_\mathrm{max}}$ that lie in a Heisenberg coordinate patch (see \autoref{defn:COORDINATES}) containing $p$. We write
\begin{equation}\label{eqn:COORDINATES}
\begin{alignedat}{3}
z &= (z_0, \dotsc, z_{m-1}), \quad && \theta = \Rep z_0,  \quad &&u = (z_1, \dotsc, z_{m-1}),\\
w &= (w_0, \dotsc, w_{m-1}), \quad && \varphi = \Rep w_0, \quad && v = (w_1, \dotsc, w_{m-1}).
\end{alignedat}
\end{equation}
Note that if $z \in \partial M_\tau$, then $z = (\theta, u)$.

We record the Taylor expansions of the boundary defining function \eqref{eqn:DEFININGFUNCTION} and the Hamilton flow \eqref{eqn:GTTAU} on the boundary in these coordinates. Detailed computations in the setting of a real hypersurface in a complex manifold whose induced CR structure is strongly pseudoconvex (as is the case of $\partial M_\tau \subset M_\C$) are found in \cite[Section~18]{FollandStein}.

\begin{Lemma}\label{Polarizationexpansion}
    Let $\phi_\tau$ be the defining function \eqref{eqn:DEFININGFUNCTION} of the Grauert tube boundary. Denote by $\phi_\tau(z,\overline{w})$ be its analytic extension obtained by polarization. Then, in Heisenberg coordinates near $p$ in $M_\tau$, we have
    \begin{align}
    \phi_{\tau}(z, \overline{w} ) = \frac{i}{2}( z_0 - \overline{w}_0 ) + \sum_{j = 1}^{m - 1} z_j \overline{w}_j + R( z,\overline{w}).
\end{align}
The remainder \(R\) takes the form
\begin{align} \label{phitau remainder}
    R( z, \overline{w}) = R^{\phi_{\tau}}( z_0, u, \overline{w}_0, \overline{v} ) =R_2(z_0, \overline{w}_0) + R_2(z_0,  u, \overline{w}_0 ,\overline{v} ) + R_3(u, \overline{v}),
\end{align}
where 
\begin{itemize}
\item $R_2(z_0, \overline{w}_0)$ contains only of mixed terms of the form \(z_0^{\alpha} \overline{w}_0^{\beta} \) with \(|\alpha| + |\beta| \geq 2\).
\item $R_2(z_0,  u, \overline{w}_0 ,\overline{v} )$ contains only of mixed terms of the form \(z_0^{\alpha} \overline{v}^{\beta}\) or \( \overline{w}_0^{\alpha} u^{\beta}  \) with \(|\alpha| + |\beta| \geq 2\).
\item $R_3(u, \overline{v})$ contains only of mixed terms of the form \(u^{\alpha} \overline{v}^{\beta} \) with \( |\alpha| + |\beta| \geq 3\).
\end{itemize}
\end{Lemma}

\begin{proof}
Direct computation (see \cite[(18.4)]{FollandStein}) yields
\begin{align}\label{Taylor Expansion of phitau}
    \phi_{\tau}(z) = \phi_\tau(z_0, u) = - \Imp z_0 + \abs{u}^{2} + O( \abs{z_0} \abs{u} + \abs{z_0}^{2} + \abs{u}^{3}),
\end{align}
so
    \begin{align}
        d \phi_{\tau}\big|_{p} = - \Imp d z_0 \big|_{p} \quad \text{and} \quad \frac{\partial^{2} \phi_{\tau}}{\partial z_j \partial \overline{z}_k}\bigg|_p = \delta_{jk} \quad (1 \leq j,k \leq m - 1).
    \end{align}
Combined with the near-diagonal Taylor expansion
\begin{align}
    \phi_{\tau}(p + h, p + k) = \sum_{\alpha, \beta} \frac{\partial^{\alpha + \beta} \phi_{\tau}}{\partial z^{\alpha} \partial\bar{z} ^{\beta}}(p) \frac{h^{\alpha}}{\alpha !} \frac{\overline{k}^{\beta}}{\beta !}
\end{align}
of $\phi_\tau$, we obtain the desired result.
\end{proof}

\begin{Rem}\label{rem:HEISENBERG}
Note that the formula \eqref{Taylor Expansion of phitau} for the defining function implies $\partial M_\tau$ is highly tangent at $p$ to the hypersurface $\Imp z_0 = \sum_{k=1}^{m-1} \abs{z_k}^2$, which is the geometric model \eqref{eqn:MODELHEISENBERG} for the Heisenberg group.
\end{Rem}

\begin{Lemma}\label{flowtaylorexpansion}
   Let $\gtc{t}$ be the Hamilton flow \eqref{eqn:GTTAU} of the Grauert tube function on $\partial M_\tau$. Then, in Heisenberg coordinates near $p$ in $\partial M_\tau$, we have
    \begin{equation}
        \gtc{t}(z) = \gtc{t}(\theta, u) = \p[\Big]{ \theta + 2 \tau t + t \cdot O^1 + O( t^{2} ), 	u + t \cdot O^1 + O( t^{2})},
    \end{equation} 
		where \(O^{1}\) denotes the Heisenberg-type order \eqref{eqn:HEISENBERGORDER}.
\end{Lemma}
\begin{proof}
     By \cite[Theorem~18.5]{FollandStein}, the Reeb vector field is of the form
     \begin{align}
         \frac{1}{2\tau} H_{\sqrt{\rho} } =  \frac{\partial }{\partial \theta} + \sum_{j = 1}^{m - 1} O^{1} \frac{\partial }{\partial z_j} + \sum_{j = 1}^{m - 1} O^{1}\frac{\partial }{\partial \overline{z}_j } + O^{1} \frac{\partial }{\partial  \theta}.
     \end{align}
 The Lemma then follows by Taylor expanding \(\gtc{t}\) near \(t = 0\).
\end{proof}

\section{Proof of main results}

We briefly outline the proof of \autoref{Scaling Theorem} before delving into the computations. As indicated in a subsequent Section, the same techniques work almost verbatim for proving \autoref{Scaling Theorem 2}.

Since we are interested in scaling asymptotics, variables are rescaled according to their Heisenberg-type order (cf~\eqref{eqn:HEISENBERGORDER}). Namely, given $(\theta, u) \in \partial M_\tau$ in Heisenberg coordinates near $p \in \partial M_\tau$ (recall \autoref{defn:COORDINATES} and \eqref{eqn:COORDINATES}), we consider
\begin{equation}
(\theta, u) \mapsto \p[\bigg]{\frac{\theta}{\lambda}, \frac{u}{\sqrt{\lambda}}} \in \partial M_\tau.
\end{equation}

We rewrite the (rescaled) spectral localization kernel of \eqref{eqn:PCL} in terms of the kernel of a ``dynamical Toeplitz operator'' using \autoref{Dynamical Toeplitz}. Then, substituting the Boutet de Monvel--Sj\"{o}strand parametrix \eqref{Parametrix for Szego Kernel} for the Szeg\H{o} projectors, we are left with an oscillatory integral. Taylor expanding the phase function using \autoref{Polarizationexpansion}, \autoref{flowtaylorexpansion}, and then applying the method of stationary phase complete the proof.
 
\subsection{Proof of \autoref{Scaling Theorem}}

From \eqref{eqn:PCL} and \autoref{Dynamical Toeplitz}, we have
\begin{equation}
 \Pi_{\chi,\lambda} = \int_\R \hat{\chi}(t) \Pi_\tau \hat{\sigma}_{t,\tau} (\gtc{t})^*\Pi_\tau\,dt \quad \text{modulo a smoothing Toeplitz operator}.
\end{equation}
Fix a point \(p \in \pmt\) and a Heisenberg coordinate chart centered at \(p = 0\). Fix two points \(( \theta, u ), ( \varphi, v) \in \partial M_\tau\) in Heisenberg coordinates. Then, substituting the parametrix \eqref{Parametrix for Szego Kernel} for each instance of $\Pi_\tau$ above and composing the resulting kernels, we arrive at the oscillatory integral representation
\begin{align}\label{eqn:OSCINTORIGINAL}
    \Pi_{\chi, \lambda}\bigg( \frac{\theta}{\lambda}, \frac{u}{\sqrt{\lambda} }, \frac{\varphi}{\lambda}, \frac{v}{\sqrt{\lambda} } \bigg) \sim \int_{\R \times \partial M_\tau \times \R^+ \times \R^+} e^{i \lambda \Psi} A \,\quaddifferential,
\end{align}
in which the phase $\Psi$ and the amplitude $A$ are given by
\begin{equation}\label{eqn:AMPA}
\begin{aligned}
\Psi &= -t + \frac{1}{\lambda} \sigma_2 \psi_\tau\bigg( \bigg( \frac{\theta}{\lambda}, \frac{u}{\sqrt{\lambda} } \bigg), w \bigg) + \frac{1}{\lambda} \sigma_1   \psi_{\tau}\bigg( \gtc{t}(w), \bigg( \frac{\varphi}{\lambda}, \frac{v}{\sqrt{\lambda} } \bigg) \bigg),\\
A &= \hat{\chi}(t) s\bigg( \tl, \ul, w, \sigma_2 \bigg)s\bigg( \gtc{t}(w), \phil, \vl, \sigma_1 \bigg) \sigma_{t, \tau}(w, \sigma_1  ).
\end{aligned}
\end{equation}

\begin{Rem}\label{rem:A}
Recall $\sigma_{t,\tau}$ is the unitarization symbol of order zero from \autoref{Dynamical Toeplitz}, $s = s(z,w,\sigma)$ is the symbol of order $m-1$ in the Boutet de Monvel--Sj\"{o}strand parametrix \eqref{Parametrix for Szego Kernel}, and $\sigma = \sigma_j \in \R^+$ are parameters of integration in the parametrix. We do not yet put the spatial variable $x \in \partial M_\tau$ of integration in Heisenberg coordinates.
\end{Rem}

Make the change-of-variables \(\sigma_j \mapsto \lambda \sigma_j\). Homogeneity of the symbols implies
\begin{align}\label{Partial Szego Kernel Scaled}
    \Pi_{\chi, \lambda}\bigg( \frac{\theta}{\lambda}, \frac{u}{\sqrt{\lambda} }, \frac{\varphi}{\lambda}, \frac{v}{\sqrt{\lambda} } \bigg) \sim \lambda^{2m}\int_{\R \times \partial M_\tau \times \R^+ \times \R^+} e^{i \lambda \tilde{\Psi}} \tilde{A} \,\quaddifferential,
\end{align}
in which the phase $\tilde{\Psi}$ and the amplitude $\tilde{A}$ are given by
\begin{equation}\label{eqn:AMPANDPHASE}
\begin{aligned}
\tilde{\Psi} &= -t +  \sigma_2 \psi_\tau\bigg( \bigg( \frac{\theta}{\lambda}, \frac{u}{\sqrt{\lambda} } \bigg), w \bigg) + \sigma_1   \psi_{\tau}\bigg( \gtc{t}(w), \bigg( \frac{\varphi}{\lambda}, \frac{v}{\sqrt{\lambda} } \bigg) \bigg),\\
\tilde{A} &= \lambda^{-2m + 2} A.
\end{aligned}
\end{equation}

\begin{Rem}\label{rem:ATILDE}
It follows from \autoref{rem:A} that \(\tilde{A}\) is a symbol of order \(0\).
\end{Rem}

We begin by localizing in \((w,t) \in \partial M_\tau \times \R\). Fix $C > 0$ and \(0 < \delta \ll 1\). Set
\begin{equation}\label{eqn:CUTOFF}
\begin{aligned}
    V_{\lambda} & = \set*{(w,t) : \textstyle  \max\Big\{ d\big(w, \big(\frac{\theta}{\lambda}, \frac{u}{\sqrt{\lambda} }\big) \big), d\big( \gtc{t}(w), \big( \frac{\varphi}{\lambda}, \frac{v}{\sqrt{\lambda} } \big) \big) \Big\}< \frac{2C}{3}\lambda^{\delta - \frac{1}{2}}},\\
    W_{\lambda} & = \set*{(w,t)  : \textstyle \max \Big\{d\big(w, \big(\frac{\theta}{\lambda}, \frac{u}{\sqrt{\lambda} } \big)\big), d\big( \gtc{t}(w), \big( \frac{\varphi}{\lambda}, \frac{v}{\sqrt{\lambda} } \big) \big)\Big\}  > \frac{C}{2}\lambda^{\delta - \frac{1}{2}}}.
\end{aligned}
\end{equation}
Let \(\{\varrho_{\lambda}, 1 - \varrho_{\lambda}\} \) be a partition of unity subordinate to the cover $\{V_\lambda, W_\lambda\}$ and decompose the integral \eqref{Partial Szego Kernel Scaled} into
\begin{align}
     \Pi_{\chi, \lambda}\bigg( \frac{\theta}{\lambda}, \frac{u}{\sqrt{\lambda} }, \frac{\varphi}{\lambda}, \frac{v}{\sqrt{\lambda} } \bigg) & \sim I_1 + I_2,\\
		I_1 &= \lambda^{2m}\int e^{i \lambda \tilde{\Psi}}  \varrho_{\lambda}( t,w )  \tilde{A} \, \quaddifferential ,\\
		I_2 &= \lambda^{2m}\int e^{i \lambda \tilde{\Psi}} ( 1 - \varrho_{\lambda}( t,w ) ) \tilde{A} \, \quaddifferential.
\end{align}

\begin{Lemma}
    We have $I_2 =  O( \lambda^{ - \infty})$.
\end{Lemma}

\begin{proof}
    By definition of $W_\lambda$, on the support of \(1 - \varrho_{\lambda}\) either
    \begin{align}
       \abs{d_{\sigma_2} \tilde{\Psi}} &=  \abs*{\psi_{\tau} \bigg(  \bigg( \frac{\theta}{\lambda}, \frac{u}{\sqrt{\lambda} } \bigg),w \bigg)} \geq 2 D\bigg(\bigg( \frac{\theta}{\lambda}, \frac{u}{\sqrt{\lambda} } \bigg), w \bigg) \geq C' \lambda^{2\delta - 1}
			\end{align}
			or
			\begin{align}
       \abs{d_{\sigma_1} \tilde{\Psi}} &=  \abs*{\psi_{\tau} \bigg(  \bigg( \gtc{t}(w),\frac{\phi}{\lambda}, \frac{v}{\sqrt{\lambda} } \bigg) \bigg)} \geq 2 D\bigg(\bigg( \gtc{t}(w), \frac{\phi}{\lambda}, \frac{v}{\sqrt{\lambda} } \bigg)\bigg) \geq C' \lambda^{2 \delta - 1}
     \end{align}
where \(D\) is the Calabi diastasis \eqref{eqn:DIASTASIS} and the inequalities follow from \eqref{Diastasis inequality}. Repeated integration by parts in \(\sigma_1\) or \(\sigma_2\) as appropriate completes the proof.
\end{proof}

So far, we have shown
\begin{equation}\label{eqn:OSCIINT}
\Pi_{\chi, \lambda}\bigg( \frac{\theta}{\lambda}, \frac{u}{\sqrt{\lambda} }, \frac{\varphi}{\lambda}, \frac{v}{\sqrt{\lambda} } \bigg) \sim  \lambda^{2m}\int e^{i \lambda \tilde{\Psi}}  \varrho_{\lambda} \tilde{A} \, \quaddifferential,
\end{equation}
with $\tilde{\Psi},\tilde{A}$ as in \eqref{eqn:AMPANDPHASE} and $\varrho_\lambda$ a smooth cutoff to a neighborhood of $V_\lambda$ as in \eqref{eqn:CUTOFF}. The next step is to Taylor expand the phase function and absorb the remainder into the amplitude. Since the expansion \eqref{Taylor Expansion of phitau} is stated in terms of ambient coordinates on the Grauert tube (as opposed to coordinates on the boundary), we need to introduce some notation. Recall in Heisenberg coordinates, we have $z = (z_0, \dotsc, z_{m-1})$ in terms of coordinates in the ambient Grauert tube with the property that if in addition $z \in \partial M_\tau$, then $z = (\Rep z_0, z_1, \dotsc, z_{m-1})$ in terms of coordinates on the Grauert tube boundary. Here, we have a point $(\frac{\theta}{\lambda}, \frac{u}{\sqrt{\lambda} })$ in coordinates on $\partial M_\tau$. In terms of the ambient coordinates on \(M_{\tau_\mathrm{max}}\), we denote the same point by \(( \Theta(\lambda), \frac{u}{\sqrt{\lambda} })\), where \(\Rep \Theta(\lambda) = \frac{\theta}{\lambda}\) and  \(\Imp \Theta(\lambda)\) is determined by \(\phi_{\tau}( \Theta(\lambda), \frac{u}{\sqrt{\lambda} } ) = 0\).

\begin{Lemma}\label{lem:REMAINDER}
    We have
	\begin{align}
    \Imp \Theta(\lambda) = \frac{1}{\lambda}\sum_{j = 1}^{m - 1} \abs{u_j}^{2} + O(\lambda^{ - \frac{3}{2}}).
\end{align}
In particular, $\abs{\Imp \Theta(\lambda) } = O(\lambda^{-1}) = \abs{\Rep \Theta(\lambda)}$.
\end{Lemma}
\begin{proof}
    The implicit function theorem applied to \(\phi_{\tau} ( \Theta(\lambda), \frac{u}{\sqrt{\lambda} } ) = 0\) shows \(\Imp \Theta = {O}(\frac{1}{\sqrt{\lambda} })\). Substituting $\abs{z_0} = \abs{\Theta(\lambda)} = {O}(\frac{1}{\sqrt{\lambda} })$ in the remainder term of \eqref{Taylor Expansion of phitau} shows \(\Imp \Theta = {O}(\frac{1}{\lambda})\). Since \(\Rep \Theta = O(\frac{1}{\lambda})\), we can in fact substitute $\abs{z_0} = \abs{\Theta(\lambda)} = {O}(\frac{1}{\lambda })$ in the remainder term of \eqref{Taylor Expansion of phitau}, proving the Lemma.
\end{proof}

We write the spatial variable $w$ of integration as
\begin{alignat}{2}
w &= (w_0, w')  &&\quad \text{with $w' = (w_1, \dotsc, w_{m-1})$ in coordinates on $M_\tau$},\\
w &= (\Rep w_0, w') &&\quad \text{in coordinates on $\partial M_\tau$},\\
w(t) &= \gtc{t}(w).
\end{alignat}
We are now ready to Taylor expand $i\psi_\tau = \phi_\tau$ that appears in the phase \eqref{eqn:AMPANDPHASE}. By \autoref{Polarizationexpansion}, Taylor expansions in the ambient coordinates take the form

\begin{align}
   \begin{split}\label{First psitau expansion}
	i\psi_{\tau}\bigg(\bigg(\Theta(\lambda), \frac{u}{\sqrt{\lambda} } \bigg), w \bigg) &= \frac{i}{2}\left( \Theta(\lambda) - \overline{w}_0   \right) + \frac{1}{\sqrt{\lambda} }\sum_{j = 1}^{m - 1}u_j \overline{w}_j \\
		&\quad + R\bigg(\Theta(\lambda), \overline{w}_0, \frac{u}{\sqrt{\lambda} }, \overline{w}' \bigg),
		\end{split}\\
		\begin{split}\label{Second psitau expansion}
		    i\psi_{\tau}\bigg( \gtc{t}(w), \bigg( \Phi(\lambda), \frac{v}{\sqrt{\lambda} } \bigg) \bigg) &= \frac{i}{2}( w_0(t) -  \overline{\Phi}( \lambda ) ) + \frac{1}{\sqrt{\lambda} }\sum_{j = 1}^{m - 1}w_j(t) \overline{v}_j \\
				&\quad  + R\bigg( w_0(t), \overline{\Phi}( \lambda)   ,w'(t), \frac{\overline{v} }{\sqrt{\lambda} }  \bigg). 
				\end{split}
\end{align}
\begin{Rem}
The obvious notation $(\Phi( \lambda ), \frac{v}{\sqrt{\lambda} } )$ denotes the ambient coordinates of $(\frac{\phi}{\lambda}, \frac{v}{\sqrt{\lambda}})$. The form of the remainder $R$ is explicitly stated in \autoref{Polarizationexpansion}.
\end{Rem}

To apply stationary phase, the variables $t \in \R$ and $w \in \partial M_\tau$ need to be rescaled:
\begin{align}\label{eqn:RESCALE}
t \mapsto \frac{t}{\sqrt{\lambda}} \quad \text{and} \quad \text{$(\Rep w_0, w') \mapsto \p[\bigg]{\frac{\Rep w_0}{\sqrt{\lambda}}, \frac{w'}{\sqrt{\lambda}}}$}.
\end{align}
We write $w_0(\lambda)$ for the image of $w_0$ under this rescaling, so that $\Rep w_0(\lambda) = \frac{\Rep w_0}{\sqrt{\lambda}}$ and $\Imp w(\lambda)$ is determined by $\phi_\tau(w_0(\lambda), \frac{w'}{\sqrt{\lambda}}) = 0$.
\begin{Rem}
Since $w$ is only an integration variable, it is not necessary that rescaling of $w$ respects Heisenberg-type order. Indeed, the uniform $\sqrt{\lambda}$-rescaling \eqref{eqn:RESCALE} is needed for stationary phase.
\end{Rem}

 Under the change-of-variables \eqref{eqn:RESCALE}, the expansion \eqref{First psitau expansion} turns into
\begin{align}
    i\psi_{\tau}\bigg( \bigg( \Theta(\lambda), \frac{u}{\sqrt{\lambda} } \bigg), w( \lambda) \bigg) &= \frac{i}{2}( \Theta( \lambda) - \overline{w}_0( \lambda ) ) + \frac{1}{{\lambda} } \sum_{j = 1}^{m - 1}u_j \overline{w}_j\\
		&\quad+ R\bigg( \Theta( \lambda ), \overline{w}_0 ( \lambda ), \frac{u}{\sqrt{\lambda}}, \frac{\overline{w}' }{\sqrt{\lambda} } \bigg).
\end{align}
Multiplying through by $\lambda$ and using \autoref{Polarizationexpansion}, \autoref{lem:REMAINDER} to compute the remainder, we find
\begin{align}\label{lambdapsitau full expansion}
   i\lambda \psi_{\tau}\bigg( \bigg(\Theta(\lambda), \frac{u}{\sqrt{\lambda} }\bigg), w( \lambda) \bigg)  = - \sqrt{\lambda} \frac{i}{2} \Rep w_0 + \tilde{R}\bigg(\frac{\theta}{\lambda}, \frac{\Rep w_0}{\sqrt{\lambda} },\frac{u}{\sqrt{\lambda} },\frac{w'}{\sqrt{\lambda} }\bigg),
\end{align}
where 
\begin{align}
    \tilde{R} = \frac{i}{2}\theta - \frac{\abs{u}^{2}}{2} - \frac{\abs{w}^{2}}{2} + u \cdot \overline{w} + \lambda Q\bigg( \frac{\theta}{\lambda}, \frac{u}{\sqrt{\lambda} } , \frac{\Rep w_0}{\sqrt{\lambda} }, \frac{w'}{\sqrt{\lambda} }  \bigg) \label{Full R_1 remainder}
\end{align}
and \(Q\) has the following asymptotic expansion
\begin{align}
    Q\bigg( \frac{\theta}{\lambda}, \frac{u}{\sqrt{\lambda} } , \frac{\Rep w_0}{\sqrt{\lambda} }, \frac{w'}{\sqrt{\lambda} }  \bigg) & \sim \sum_{j = 1} \sum_{j + k + l = 2}P_{j,k,l}\bigg( \frac{u}{\sqrt{\lambda} }, \frac{\Rep w_0}{ \sqrt{\lambda} }, \frac{w'}{\sqrt{\lambda} } \bigg) \\
		&\quad +  \sum_{j + \frac{k + l + m }{2} \geq 3}P_{j,k,l,m}\bigg( \frac{\theta}{\lambda}, \frac{u}{\sqrt{\lambda} }, \frac{ \Rep w_0}{\sqrt{\lambda} }, \frac{w'}{\sqrt{\lambda} } \bigg).
\end{align}
Here, $P_{j,k,l}$ is a polynomial in three variables that is of degree $j, k, l$ in the first, second, and third slot, respectively. (The polynomial $P_{j,k,l,m}$ is similarly defined.) 

Similarly, using in addition \autoref{flowtaylorexpansion}, under the change-of-variables \eqref{eqn:RESCALE} the expansion \eqref{Second psitau expansion} turns into
\begin{equation}\label{lambdapsitau full expansion 1}
\begin{aligned}
   i \lambda \psi_{\tau}\bigg(\gtc{\frac{t}{\sqrt{\lambda} }}( w( \lambda) ),\bigg( \Phi(\lambda), \frac{v}{\sqrt{\lambda} } \bigg)\bigg) &= \sqrt{\lambda} \frac{i}{2}\left( \Rep w_0 + 2 \tau t \right) \\
	&\quad + \tilde{S}\bigg( \frac{\varphi}{\lambda},  \frac{\Rep w_0}{\sqrt{\lambda} }, \frac{t}{\sqrt{\lambda} }, \frac{v}{\sqrt{\lambda} }, \frac{w'}{\sqrt{\lambda} } \bigg),
 \end{aligned}
\end{equation}
 where
 \begin{align}
     \tilde{S} = - \frac{i}{2}\varphi - \frac{\abs{v}^{2}}{2} - \frac{ \abs{w}^{2}}{2} + \overline{v} \cdot w + \lambda  T\bigg( \frac{\varphi}{\lambda}, \frac{v}{\sqrt{\lambda} },   \frac{t}{\sqrt{\lambda} }, \frac{\Rep w_0}{\sqrt{\lambda} }, \frac{w'}{\sqrt{\lambda} } \bigg) \label{Full S_1 remainder}
 \end{align}
and \(T\) has the following asymptotic expansion
\begin{align}
      T\bigg( \frac{\varphi}{\lambda}, \frac{v}{\sqrt{\lambda} },   \frac{t}{\sqrt{\lambda} }, \frac{\Rep w_0}{\sqrt{\lambda} }, \frac{w'}{\sqrt{\lambda} } \bigg)  &\sim  \sum_{ i + j > 0}\sum_{i + j + k + l = 2 }P_{i,j,k,l}\bigg( \frac{t}{\sqrt{\lambda} }, \frac{\Rep w_0}{ \sqrt{\lambda} },  \frac{v}{\sqrt{\lambda} }, \frac{w'}{\sqrt{\lambda} } \bigg) \\
		&\quad + \sum_{i + \frac{j + k + l + m}{2}}P_{i,j,k,l,m}\bigg( \frac{\varphi}{\lambda}, \frac{v}{\sqrt{\lambda} }, \frac{t}{\sqrt{\lambda} }, \frac{\Rep w_0}{\sqrt{\lambda} }, \frac{w'}{\sqrt{\lambda} } \bigg).
\end{align}

It follows from \eqref{lambdapsitau full expansion 1} and \eqref{lambdapsitau full expansion} that under the change-of-variables \eqref{eqn:RESCALE}, the integral \eqref{eqn:OSCIINT} becomes an oscillatory integral with parameter $\sqrt{\lambda}$:
\begin{equation}\label{eqn:OSCIINT2}
\Pi_{\chi, \lambda}\bigg( \frac{\theta}{\lambda}, \frac{u}{\sqrt{\lambda} }, \frac{\varphi}{\lambda}, \frac{v}{\sqrt{\lambda} } \bigg) \sim  \lambda^{m}\int e^{i \sqrt{\lambda} \dtilde{\Psi}}   \dtilde{A} \, \quaddifferential,
\end{equation}
where
\begin{equation}\label{eqn:AMPPHASE2}
\begin{aligned}
\dtilde{\Psi} &= - t - \frac{\sigma_2}{2}\Rep w_0 +\frac{\sigma_1}{2}(\Rep w_0 + 2 \tau t ), \\
\dtilde{A} &= e^{\sigma_2\tilde{R} + \sigma_1 \tilde{S}}  \varrho_{\lambda} \tilde{A}\bigg( \frac{\theta}{\lambda} , \frac{u}{\sqrt{\lambda} },\frac{\varphi}{\lambda},  \frac{v}{\sqrt{\lambda} } ,  \frac{\Rep w_0}{\sqrt{\lambda} }, \frac{w'}{\sqrt{\lambda} } ,\frac{t}{\sqrt{\lambda} } ,    \sigma_1, \sigma_2  \bigg) J ( w(\lambda) ),
\end{aligned}
\end{equation}
with $J(\eta)$ the volume density in Heisenberg coordinates.

\begin{Rem}
In \eqref{eqn:RESCALE}, the spatial rescaling introduces a factor of $\lambda^{- m + 1/2}$, while the time rescaling introduces a factor of $\lambda^{-1/2}$. Together with the overall prefactor of $\lambda^{2m}$ in \eqref{eqn:OSCIINT}, we obtain the $\lambda^m$ prefactor in \eqref{eqn:OSCIINT2}. The explicit descriptions of $\tilde{R}$ and $\tilde{S}$ imply that $e^{\sigma_2 \tilde{R} + \sigma_1 \tilde{S}}$ may be absorbed into the amplitude.
\end{Rem}

We make one final localization argument. Let \(\set{\eta, 1 - \eta}\) be a partition of unity subordinate to the cover
\begin{equation}
\set*{(\sigma_1, \sigma_2) : 0 < \sigma_1, \sigma_2  < \frac{2}{\tau}} \quad \text{and} \quad \set*{(\sigma_1, \sigma_2) : \sigma_1, \sigma_2 > \frac{3}{2 \tau} }.
\end{equation}
Decompose \eqref{eqn:OSCIINT2} into two integrals:
\begin{align}
\Pi_{\chi, \lambda}\bigg( \frac{\theta}{\lambda}, \frac{u}{\sqrt{\lambda} }, \frac{\varphi}{\lambda}, \frac{v}{\sqrt{\lambda} } \bigg) &\sim I_1' + I_2',\\
I_1' &= \lambda^{m}\int e^{i \lambda \dtilde{\Psi}}  \eta(\sigma_1, \sigma_2) \dtilde{A} \, \quaddifferential,\\
I_2' &= \lambda^{m}\int e^{i \lambda \dtilde{\Psi}}  (1 - \eta(\sigma_1, \sigma_2))  \dtilde{A} \, \quaddifferential,
\end{align}
with $\dtilde{A}$ and $\dtilde{\Psi}$ as in \eqref{eqn:AMPPHASE2}. 
\begin{Lemma}
We have $I_2' = O(\lambda^{-\infty})$.
\end{Lemma}
\begin{proof}
Notice that 
\begin{align}
    \abs*{ \nabla_{\Rep w_0,t} \dtilde{\Psi}}^2 \geq \bigg( \frac{\sigma_1}{2} - \frac{\sigma_2}{2} \bigg) ^{2} + ( \tau\sigma_1 - 1 )^{2} \ge \frac{1}{4}
\end{align}
on the support of \(1 - \eta\). Thus, the Lemma follows from repeated integration by parts in \((\Rep w_0, t)\).
\end{proof}

We have finally reduced the spectral localization kernel to the oscillatory integral
\begin{equation}\label{eqn:OSCINTFINAL}
\Pi_{\chi, \lambda}\bigg( \frac{\theta}{\lambda}, \frac{u}{\sqrt{\lambda} }, \frac{\varphi}{\lambda}, \frac{v}{\sqrt{\lambda} } \bigg) \sim  \lambda^{m} \int_{\R \times \R^+ \times \R^+ \times \R \times \C^{m-1}} e^{i \sqrt{\lambda} \dtilde{\Psi}} B \,dw' d(\Rep w_0) d\sigma_1 d\sigma_2 dt,
\end{equation}
with phase and amplitude
\begin{align}
\dtilde{\Psi} &= - t - \frac{\sigma_2}{2}\Rep w_0 +\frac{\sigma_1}{2}(\Rep w_0 + 2 \tau t ), \\
\dtilde{A} &=e^{\sigma_2\tilde{R} + \sigma_1 \tilde{S}}  \eta \varrho_{\lambda} \hat{\chi} \tilde{A} J.
\end{align}
\begin{Rem}
The cutoff functions $\varrho_\lambda, \eta, \hat{\chi}$ localize $B$ to a compact region. The $\lambda^{m-1}$ factor in $B$ comes from comparing the amplitude $\dtilde{A}$ in the integral \eqref{eqn:OSCIINT2} (which already has an overall factor of $\lambda^m$ from the change-of-variables $\sigma_j \mapsto \lambda \sigma_j$ and \eqref{eqn:RESCALE}) with the amplitude $A$ in \eqref{eqn:OSCINTORIGINAL}.
\end{Rem}

We are now in the position to integrate with respect to \(\Rep w_0, \sigma_1, \sigma_2, t\) using the method of stationary phase, thereby reducing \eqref{eqn:OSCINTFINAL}  to a Gaussian integral over \(\C^{m - 1}\). We note the following derivatives:
\begin{align}
     \partial_{\sigma_2} \dtilde{\Psi} & = - \frac{1}{2} \Rep w_0, &   \partial_{\sigma_1} \dtilde{\Psi} &= \frac{1}{2}( \Rep w_0 + 2 \tau t ),\\
     \partial_{t} \dtilde{\Psi} & = - 1 + \tau{\sigma_1}, & \partial_{\Rep w_0} \dtilde{\Psi}  &= - \frac{\sigma_2}{2} + \frac{\sigma_1}{2}.
 \end{align}
The critical set of the phase is the point \(C =\{ \Rep w_0 = 0, t = 0, \sigma_1 = \sigma_2 = \frac{1}{\tau}\} \). The Hessian matrix and its inverse at the critical point are
\begin{align}
   \dtilde{\Psi}''_{C} &=  \p*{\begin{array}{c|cccc}  & t & \sigma_1 & \sigma_2 & \Rep w_0 \\ \hline  t  & 0 & \tau & 0 & 0\\ \sigma_1 & \tau & 0 & 0 & \frac{1}{2} \\ \sigma_2 & 0 & 0 & 0 & - \frac{1}{2} \\ \Rep w_0 & 0 &\frac{1}{2} & - \frac{1}{2} & 0\end{array}}, \quad (\dtilde{\Psi}''_C)^{-1} =      \begin{pmatrix} 0 & \frac{1}{\tau} & \frac{1}{\tau} & 0 \\
    \frac{1}{\tau} & 0 & 0 &0 \\ \frac{1}{\tau} & 0 &0 &- 2\\ 0 & 0 &- 2 & 0 \end{pmatrix} .
 \end{align}

Set
 \begin{align}
     L_{\dtilde{\Psi}} = \ang*{(\dtilde{\Psi}''_C)^{-1} D, D} =  \frac{2}{\tau}\partial_{\sigma_1}\partial_{t} + \frac{2}{\tau} \partial_{\sigma_2}\partial_t - 4 \partial_{\sigma_2}\partial_{\Rep w_0}
 \end{align}
By the method of stationary phase (\cite[Theorem 7.75]{Hormanderv1}), we have
 \begin{multline} 
     \int_{\R\times \R^{ + } \times \R^{ +} \times \R}e^{i \sqrt{\lambda} \dtilde{\Psi}} \dtilde{A} \, dw' d\sigma_1 d\sigma_2dt \\ 
		= \gamma \lambda^m \sum_{j = 0}^{N-1}\lambda^{ - \frac{j}{2}}\sum_{\nu - \mu = j} \sum_{2\nu \ge 3 \mu} \frac{1}{i^j 2^\nu} L_{\dtilde{\Psi}}^\nu \bra*{e^{\mu(\sigma_2\tilde{R} + \sigma_1 \tilde{S})}  \eta \varrho_{\lambda} \hat{\chi} \tilde{A} J }_{C} +\hat{R}_N.   \label{Abstract Stationary Phase Expansion}
 \end{multline}
The leading coefficient $\gamma$ is given by
\begin{align}
     \gamma = e^{i \sqrt{\lambda} \dtilde{\Psi}|_{C}} \left( \det \left( \sqrt{\lambda} \dtilde{\Psi}''_{C}  /2 \pi i \right) \right)^{ - \frac{1}{2}} = \frac{8 \pi^{2}}{\lambda \tau}.
 \end{align}

The stationary phase remainder estimate implies
\begin{align}
\int_{\C^{m - 1}}^{} &\abs{\hat{R}_{N}}\,dw' \leq \int_{\C^{m - 1}}^{} \lambda^{ - \frac{N}{2}} \sum_{ \abs{\alpha} \leq 2 N}\sup \abs{D^{\alpha}(\eta \rho_\lambda \hat{\chi} \tilde{A} J)} \,dw' \le C_N \lambda^{-\frac{N}{2}}
\end{align}
because $\tilde{A}$ is a symbol of order zero by \autoref{rem:ATILDE}. (Here, the supremum and the derivative $D^\alpha$ are taken over $t, \sigma_1, \sigma_2, \Rep w_0$ and the integral is with respect to the remaining variable $w'$.)

It remains to integrate the asymptotic expansion \eqref{Abstract Stationary Phase Expansion} in $w'$. The left-hand side is \eqref{eqn:OSCINTFINAL}; the right-hand side can be integrated term-by-term thanks to the estimate above.  After substituting expressions \eqref{Full R_1 remainder} and \eqref{Full S_1 remainder} and performing a Taylor expansion near the \(p = ( 0,0 ) \in \R \times \C^{m - 1}\), we find an asymptotic expansion for the kernel in increasing half-powers of \(\lambda \):
\begin{align}
\Pi_{\chi, \lambda}\bigg( \frac{\theta}{\lambda}, \frac{u}{\sqrt{\lambda} }, \frac{\varphi}{\lambda}, \frac{v}{\sqrt{\lambda} } \bigg) &\sim    \frac{8 \pi^2 \lambda^{m - 1}}{\tau} \int_{\C^{m - 1}} \bigg[e^{\frac{1}{\tau}\big( \frac{i}{2}( \theta - \varphi ) - \frac{|u|^{2}}{2} - \frac{|v|^{2}}{2} - |w'|^{2} + u \cdot \overline{w}'  + w' \cdot \overline{v}  \big)}\\
&\quad \times\sum_{j \geq 0}C_j\lambda^{ - \frac{j}{2}}P_{j}( \theta, u, \varphi, v, w' )\bigg]\,dw'\\
& = \frac{8 \pi^{2} \lambda^{m - 1}}{\tau} e^{\frac{i( \theta - \varphi )}{2\tau} - \frac{|u|^{2}}{2\tau} - \frac{|v|^{2}}{2\tau} +  \frac{v \cdot \overline{u}}{\tau}} \int_{\C^{m - 1}} \bigg[e^{\frac{1}{\tau} \big(- |z|^{2} + z \cdot ( \overline{v} - \overline{u}  )\big)}\\
&\quad \times  \sum_{j \ge 0}C_j{\lambda} ^{- \frac{j}{2}} {P_j }( \theta, u, \varphi, v, w' )\bigg]\, dz .
\end{align}
The equality follows from the change-of-variables \(w = z + u\). The principal term in the series is given by
\begin{equation}
P_0 =  \tau ^{ -2(m - 1)}s_0(p,p)^{2}\sigma_{0, \tau}(p)J(p).
\end{equation}
Since 
\begin{align}
    \int_{\C^{m - 1}}^{}e^{ \frac{1}{\tau} \big( -  |z|^{2} +  z \cdot ( \overline{v} - \overline{u}  ) \big)}\,dz & = ( \tau \pi )^{m - 1},
\end{align}
we conclude 
\begin{align}
\Pi_{\chi, \lambda}\bigg( \frac{\theta}{\lambda}, \frac{u}{\sqrt{\lambda} }, \frac{\varphi}{\lambda}, \frac{v}{\sqrt{\lambda} } \bigg) &\sim  C_m \frac{\lambda^{m-1}}{\tau^{m-2}} e^{\frac{i( \theta - \varphi )}{2\tau} - \frac{|u|^{2}}{2\tau} - \frac{|v|^{2}}{2\tau} +  \frac{v \cdot \overline{u}}{\tau}} \\
&\quad \times \left(1 + \sum_{j = 1}^{N}\lambda^{ - \frac{j}{2}}P_j(u,v,\theta,\varphi) \right)  \\
    &\quad +  \lambda^{^{m - 1 - \frac{N + 1}{2}}} R( \theta, u , \varphi, v ,\lambda ) .
\end{align}

\subsection{Proof of \autoref{Scaling Theorem 2}}
We look at the near-diagonal scaling asymptotics for the tempered spectral projection kernel \eqref{eqn:TEMPEREDPCL} under Heisenberg-type scaling. Similar to the proof of \autoref{Scaling Theorem}, we write out the kernel using \autoref{Dynamical Toeplitz Wave group} and \eqref{Parametrix for Szego Kernel}:
\begin{align}\label{Partial Wave Group}
    \Pcl\bigg( \frac{\theta}{\lambda}, \frac{u}{\sqrt{\lambda} }, \frac{\phi}{\lambda}, \frac{v}{\sqrt{\lambda} } \bigg) \sim \int_{\R \times \partial M_\tau \times \R^+ \times \R^+} e^{i \lambda \Psi} B \,\quaddifferential,
\end{align}
in which the phase $\Psi$ and the amplitude $B$ are given by
\begin{equation}
\begin{aligned}
\Psi &= -t + \frac{1}{\lambda} \sigma_2 \psi_\tau\bigg( \bigg( \frac{\theta}{\lambda}, \frac{u}{\sqrt{\lambda} } \bigg), w \bigg) + \frac{1}{\lambda} \sigma_1   \psi_{\tau}\bigg( \gtc{t}(w), \bigg( \frac{\varphi}{\lambda}, \frac{v}{\sqrt{\lambda} } \bigg) \bigg),\\
B &= \hat{\chi}(t) s\bigg( \tl, \ul, w, \sigma_2 \bigg)s\bigg( \gtc{t}(w), \phil, \vl, \sigma_1 \bigg) \sigma_{t, \tau}(w, \sigma_1  ).
\end{aligned}
\end{equation}

As we point out in \autoref{rem:ORDER}, despite identical notation, the unitarization symbol $\sigma_{t,\tau}$ in the expression for $B$ is of order $-(m-1)/2$, whereas $\sigma_{t,\tau}$ in the expression for $A$ in \eqref{eqn:AMPA} is of order zero. Hence, the oscillatory integral expression \eqref{Partial Wave Group} is exactly $\lambda^{-(m-1)/2}$ times the expression \eqref{eqn:OSCINTORIGINAL}. The rest of the computations proceed in the same manner.

\bibliographystyle{plain}
\bibliography{References}
\end{document}